\documentclass[vecarrow]{svmult}
\renewcommand{\email}[1]{\emailname: #1} % change the email address font style

\usepackage{mathptmx}       % selects Times Roman as basic font
\usepackage{helvet}         % selects Helvetica as sans-serif font
\usepackage{courier}        % selects Courier as typewriter font
\usepackage{makeidx}         % allows index generation
\usepackage{graphicx}        % standard LaTeX graphics tool
\usepackage[bottom]{footmisc}% places footnotes at page bottom

\usepackage{latexsym}
\usepackage{amsmath}
\usepackage{amsfonts}
\usepackage{amssymb}
\usepackage{bm}

\usepackage{url}
\usepackage{algorithm}
\usepackage{algorithmic}
\usepackage[misc,geometry]{ifsym}

\renewenvironment{proof}{\noindent{\itshape Proof.}}{\smartqed\qed}

\spdefaulttheorem{assumption}{Assumption}{\upshape \bfseries}{\itshape}
\spdefaulttheorem{algo}{Algorithm}{\upshape \bfseries}{\itshape}
\DeclareSymbolFont{bbold}{U}{bbold}{m}{n}
\DeclareSymbolFontAlphabet{\mathbbold}{bbold}

\usepackage{tikz,thm-restate,cleveref}
\usepackage[caption=false]{subfig}
\usepackage{adjustbox}
\newcommand{\absAinsworthGlusa}[1]{\left|#1\right|}
\newcommand{\normAinsworthGlusa}[1]{\left|\!\left|#1\right|\!\right|}
\begin{document}

\title*{Towards an Efficient Finite Element Method for the Integral Fractional Laplacian on Polygonal Domains}
\titlerunning{Efficient Finite Element Method for the Integral Fractional Laplacian}

\author{Mark Ainsworth \and Christian Glusa}

\institute{
 Mark Ainsworth (\Letter)
 \at Division of Applied Mathematics, Brown University, 182 George St, Providence, RI 02912, USA\\
 \at Computer Science and Mathematics Division, Oak Ridge National Laboratory, Oak Ridge, TN 37831, USA\\
 \email{Mark\_Ainsworth@Brown.edu}
 \and
 Christian Glusa
 \at Division of Applied Mathematics, Brown University, 182 George St, Providence, RI 02912, USA \\
 \email{Christian\_Glusa@Brown.edu}\\
 \\
 This work was supported by the MURI/ARO on ``Fractional PDEs for Conservation Laws and Beyond: Theory, Numerics and Applications'' (W911NF-15-1-0562).
 }

\maketitle

\index{Ainsworth, Mark}
\index{Glusa, Christian}

\paragraph{Dedicated to Ian H.~Sloan on the occasion of his 80th birthday.}

\abstract{
  We explore the connection between fractional order partial differential equations in two or more spatial dimensions with boundary integral operators to develop techniques that enable one to efficiently tackle the integral fractional Laplacian.
  In particular, we develop techniques for the treatment of the dense stiffness matrix including the computation of the entries, the efficient assembly and storage of a sparse approximation and the efficient solution of the resulting equations.
  The main idea consists of generalising proven techniques for the treatment of boundary integral equations to general fractional orders.
  Importantly, the approximation does not make any strong assumptions on the shape of the underlying domain and does not rely on any special structure of the matrix that could be exploited by fast transforms.
  We demonstrate the flexibility and performance of this approach in a couple of two-dimensional numerical examples.
}

\section{Introduction}
\label{sec:introduction}

Large scale computational solution of partial differential equations has revolutionised the way in which scientific research is performed.
Historically, it was generally the case that the mathematical models, expressed in the form of partial differential equations involving operators such as the Laplacian, were impossible to solve analytically, and difficult to resolve numerically.
This led to a concerted and sustained research effort into the development of efficient numerical methods for approximating the solution of partial differential equations.
Indeed, many researchers who were originally interested in applications shifted research interests to the development and analysis of numerical methods.
A case in point is Professor Ian H. Sloan who originally trained as physicist but went on to carry out fundamental research in a wide range of areas relating to computational mathematics.
Indeed, one may struggle to find an area of computational mathematics in which Sloan has not made a contribution and the topic of the present article, fractional partial differential equations, may be one of the very few.

In recent years, there has been a burgeoning of interest in the use of non-local and fractional models.
To some extent, this move reflects the fact that with present day computational resources coupled with state of the art numerical algorithms, attention is now shifting back to the fidelity of the underlying mathematical models as opposed to their approximation.
Fractional equations have been used to describe phenomena in anomalous diffusion, material science, image processing, finance and electromagnetic fluids \cite{West2016_FractionalCalculusViewComplexity}.
Fractional order equations arise naturally as the limit of discrete diffusion governed by stochastic processes \cite{MeerschaertSikorskii2012_StochasticModelsFractionalCalculus}.

Whilst the development of fractional derivatives dates back to essentially the same time as their integer counterparts, the computational methods available for their numerical resolution drastically lags behind the vast array of numerical techniques from which one can choose to treat integer order partial differential equations.
The recent literature abounds with work on numerical methods for fractional partial differential equations in one spatial dimension and fractional order temporal derivatives.
However, with most applications of interest being posed on domains in two or more spatial dimensions, the solution of fractional equations posed on complex domains is a problem of considerable practical interest.

The archetypal elliptic partial differential equation is the Poisson problem involving the standard Laplacian.
By analogy, one can consider a fractional Poisson problem involving the fractional Laplacian.
The first problem one encounters is that of how to define a fractional Laplacian, particularly in the case where the domain is compact, and a number of alternatives have been suggested.
The \emph{integral fractional Laplacian} is obtained by restriction of the Fourier definition to functions that have prescribed value outside of the domain of interest, whereas the spectral fractional Laplacian is based on the spectral decomposition of the regular Laplace operator.
In general, the two operators are different \cite{ServadeiValdinoci2014_SpectrumTwoDifferentFractionalOperators}, and only coincide when the domain of interest is the full space.

The approximation of the integral fractional Laplacian using finite elements was considered by D'Elia and Gunzburger \cite{DEliaGunzburger2013_FractionalLaplacianOperatorBounded}.
The important work of Acosta and Borthagaray \cite{AcostaBorthagaray2015_FractionalLaplaceEquation} gave regularity results for the analytic solution of the fractional Poisson problem and obtained convergence rates for the finite element approximation supported by numerical examples computed using techniques described in \cite{AcostaBersetcheEtAl2016_ShortFeImplementation2d}.

The numerical treatment of fractional partial differential equations is rather different from the integer order case owing to the fact that the fractional derivative is a non-local operator.
This creates a number of issues including the fact that the resulting stiffness matrix is dense and, moreover, the entries in the matrix are given in terms of singular integrals.
In turn, these features create issues in the numerical computation of the entries and the need to store the entries of a dense matrix, not to mention the fact that a solution of the resulting matrix equation has to be computed.
The seasoned reader will readily appreciate that many of these issues are shared by boundary integral equations arising from classical integer order differential operators \cite{Sloan1992_ErrorAnalysisBoundaryIntegralMethods,SloanSpence1988_GalerkinMethodIntegralEquations,YanSloanEtAl1988_IntegralEquationsFirstKind}.
This similarity is not altogether surprising given that the boundary integral operators are pseudo-differential operators of fractional order.

A different, integer order operator based approach, was taken by Nochetto, Ot{\'a}rola and Salgado \cite{NochettoOtarolaEtAl2015_PdeApproachToFractional} for the case of the spectral Laplacian.
Caffarelli and Silvestre \cite{CaffarelliSilvestre2007_ExtensionProblemRelatedToFractionalLaplacian} showed that the operator can be realised as a Dirichlet-to-Neumann operator of an extended problem in the half space in \(d+1\) dimensions.

In the present work, we explore the connection with boundary integral operators to develop techniques that enable one to efficiently tackle the integral fractional Laplacian.
In particular, we develop techniques for the treatment of the stiffness matrix including the computation of the entries, the efficient storage of the resulting dense matrix and the efficient solution of the resulting equations.
The main ideas consist of generalising proven techniques for the treatment of boundary integral equations to general fractional orders.
Importantly, the approximation does not make any strong assumptions on the shape of the underlying domain and does not rely on any special structure of the matrix that could be exploited by fast transforms.
We demonstrate the flexibility and performance of this approach in a couple of two-dimensional numerical examples.

\section{The Integral Fractional Laplacian and Its Weak Formulation}
\label{sec:weak-formulation}

The fractional Laplacian in \(\mathbb{R}^{d}\) of order \(s\), for \(0<s<1\) and \(d\in\mathbb{N}\), of a function \(u\) can be defined by the Fourier transform \(\mathcal{F}\) as
\begin{align*}
  \left(-\Delta\right)^{s}u = \mathcal{F}^{-1}\left[\absAinsworthGlusa{\xi}^{2s} \mathcal{F}u\right].
\end{align*}
Alternatively, this expression can be rewritten \cite{Valdinoci2009_FromLongJumpRandom} in integral form as
\begin{align*}
  \left(-\Delta\right)^{s} u\left(\vec{x}\right) = C(d,s) \operatorname{p.v.} \int_{\mathbb{R}^{d}} \; d \vec{y} ~ \frac{u(\vec{x})-u(\vec{y})}{\absAinsworthGlusa{\vec{x}-\vec{y}}^{d+2s}}
\end{align*}
where
\begin{align*}
  C(d,s) = \frac{2^{2s}s\Gamma\left(s+\frac{d}{2}\right)}{\pi^{d/2}\Gamma\left(1-s\right)}
\end{align*}
is a normalisation constant and \(\operatorname{p.v.}\) denotes the Cauchy principal value of the integral \cite[Chapter 5]{Mclean2000_StronglyEllipticSystemsBoundaryIntegralEquations}.
In the case where \(s=1\) this operator coincides with the usual Laplacian.
If \(\Omega\subset\mathbb{R}^{d}\) is a bounded Lipschitz domain, we define the \emph{integral fractional Laplacian} \(\left(-\Delta\right)^{s}\) to be the restriction of the full-space operator to functions with compact support in \(\Omega\).
This generalises the homogeneous Dirichlet condition applied in the case \(s=1\) to the case \(s\in(0,1)\).

Define the usual fractional Sobolev space \(H^{s}\left(\mathbb{R}^{d}\right)\) via the Fourier transform.
If \(\Omega\) is a sub-domain as above, then we define the Sobolev space \(H^{s}\left(\Omega\right)\) to be
\begin{align*}
   H^{s}\left(\Omega\right)&:=\left\{u\in L^{2}\left(\Omega\right) \mid \normAinsworthGlusa{u}_{H^{s}\left(\Omega\right)} < \infty\right\},
\end{align*}
equipped with the norm
\begin{align*}
  \normAinsworthGlusa{u}_{H^{s}\left(\Omega\right)}^{2}&= \normAinsworthGlusa{u}_{L^{2}\left(\Omega\right)}^{2} + \int_{\Omega}\; d \vec{x} \int_{\Omega}\; d \vec{y} \frac{\left(u(\vec{x})-u(\vec{y})\right)^{2}}{\absAinsworthGlusa{\vec{x}-\vec{y}}^{d+2s}}.
\end{align*}
The space
\begin{align*}
  \widetilde{H}^{s}\left(\Omega\right)&:=\left\{u\in H^{s}\left(\mathbb{R}^{d}\right) \mid u=0 \text{ in } \Omega^{c}\right\}
\end{align*}
can be equipped with the energy norm
\begin{align*}
  \normAinsworthGlusa{u}_{\widetilde{H}^{s}\left(\Omega\right)} := \sqrt{\frac{C(d,s)}{2}}\absAinsworthGlusa{u}_{H^{s}\left(\mathbb{R}^{d}\right)},
\end{align*}
where the non-standard factor \(\sqrt{C(d,s)/2}\) is included for convenience.
For \(s>1/2\), \(\widetilde{H}^{s}\left(\Omega\right)\) coincides with the space \(H_{0}^{s}\left(\Omega\right)\) which is the closure of \(C_{0}^{\infty}\left(\Omega\right)\) with respect to the \(H^{s}\left(\Omega\right)\)-norm.
For \(s<1/2\), \(\widetilde{H}^{s}\left(\Omega\right)\) is identical to \(H^{s}\left(\Omega\right)\).
In the critical case \(s=1/2\), \(\widetilde{H}^{s}\left(\Omega\right)\subset H^{s}_{0}\left(\Omega\right)\), and the inclusion is strict.
(See for example \cite[Chapter 3]{Mclean2000_StronglyEllipticSystemsBoundaryIntegralEquations}.)

The usual approach to dealing with elliptic PDEs consists of obtaining a weak form of the operator by multiplying the equation by a test function and applying integration by parts \cite{ErnGuermond2004_TheoryPracticeFiniteElements}.
In contrast, for equations involving the fractional Laplacian \(\left(-\Delta\right)^{s}u\), we again multiply by a test function \(v\in\widetilde{H}^{s}\left(\Omega\right)\) and integrate over \(\mathbb{R}^{d}\), and then, instead of integration by parts, we use the identity
\begin{align*}
  \int_{\mathbb{R}^{d}} \; d \vec{x} \int_{\mathbb{R}^{d}} \; d \vec{y} \frac{\left(u\left(\vec{x}\right)-u\left(\vec{y}\right)v\left(\vec{x}\right)\right)}{\absAinsworthGlusa{\vec{x}-\vec{y}}^{d+2s}}
  &=-\int_{\mathbb{R}^{d}} \; d \vec{x} \int_{\mathbb{R}^{d}} \; d \vec{y} \frac{\left(u\left(\vec{x}\right)-u\left(\vec{y}\right)v\left(\vec{y}\right)\right)}{\absAinsworthGlusa{\vec{x}-\vec{y}}^{d+2s}}.
\end{align*}
Following this approach, since both \(u\) and \(v\) vanish outside of \(\Omega\), we arrive at the bilinear form
\begin{align*}
  a(u,v)
  &= b\left(u,v\right)
    + C(d,s) \int_{\Omega} \; d\vec{x} \int_{\Omega^{c}}\; d\vec{y}  \frac{u\left(\vec{x}\right)v\left(\vec{x}\right)}{\absAinsworthGlusa{\vec{x}-\vec{y}}^{d+2s}},
    \intertext{with}
    b(u,v)
  &=\frac{C(d,s)}{2} \int_{\Omega} \; d\vec{x} \int_{\Omega}\; d\vec{y}  \frac{\left(u\left(\vec{x}\right)-u\left(\vec{y}\right)\right)\left(v\left(\vec{x}\right)-v\left(\vec{y}\right)\right)}{\absAinsworthGlusa{\vec{x}-\vec{y}}^{d+2s}},
\end{align*}
corresponding to \(\left(-\Delta\right)^{s}\) on \(\widetilde{H}^{s}\left(\Omega\right) \times \widetilde{H}^{s}\left(\Omega\right)\).
The bilinear form \(a\left(\cdot,\cdot\right)\) is trivially seen to be \(\widetilde{H}^{s}\left(\Omega\right)\)-coercive and continuous and, as such, is amenable to treatment using the Lax-Milgram Lemma.

In this article we shall concern ourselves with the computational details needed to implement the finite element approximation of problems involving the fractional Laplacian.
To this end, the presence of the unbounded domain \(\Omega^{c}\) in the bilinear form \(a\left(\cdot,\cdot\right)\) is somewhat undesirable.
Fortunately, we can dispense with \(\Omega^{c}\) using the following argument.
The identity
\begin{align*}
  \frac{1}{\absAinsworthGlusa{\vec{x}-\vec{y}}^{d+2s}} &= \frac{1}{2s} \nabla_{\vec{y}}\cdot \frac{\vec{x}-\vec{y}}{\absAinsworthGlusa{\vec{x}-\vec{y}}^{d+2s}},
\end{align*}
enables the second integral to be rewritten using the Gauss theorem as
\begin{align*}
  \frac{C(d,s)}{2s} \int_{\Omega} \; d \vec{x} \int_{\partial\Omega} \; d \vec{y} \frac{u\left(\vec{x}\right) v\left(\vec{x}\right) ~ \vec{n}_{\vec{y}}\cdot\left(\vec{x}-\vec{y}\right)}{\absAinsworthGlusa{\vec{x}-\vec{y}}^{d+2s}},
\end{align*}
where \(\vec{n}_{y}\) is the \emph{inward} normal to \(\partial\Omega\) at \(\vec{y}\), so that the bilinear form can be expressed equivalently as
\begin{align*}
  a(u,v)
  &=\frac{C(d,s)}{2} \int_{\Omega} \; d\vec{x} \int_{\Omega}\; d\vec{y}  \frac{\left(u\left(\vec{x}\right)-u\left(\vec{y}\right)\right)\left(v\left(\vec{x}\right)-v\left(\vec{y}\right)\right)}{\absAinsworthGlusa{\vec{x}-\vec{y}}^{d+2s}} \\
  &\quad+ \frac{C(d,s)}{2s} \int_{\Omega} \; d \vec{x} \int_{\partial\Omega} \; d \vec{y} \frac{u\left(\vec{x}\right) v\left(\vec{x}\right) ~ \vec{n}_{\vec{y}}\cdot\left(\vec{x}-\vec{y}\right)}{\absAinsworthGlusa{\vec{x}-\vec{y}}^{d+2s}}.
\end{align*}
As an aside, we note that the bilinear form \(b\left(u,v\right)\) represents the so-called \emph{regional fractional Laplacian} \cite{BogdanBurdzyEtAl2003_CensoredStableProcesses,ChenKim2002_GreenFunctionEstimateCensoredStableProcesses}.
The regional fractional Laplacian can be interpreted as a generalisation of the usual Laplacian with homogeneous Neumann boundary condition for \(s=1\) to the case of fractional orders \(s\in(0,1)\).
It will transpire from our work that most of the presented techniques carry over to the regional fractional Laplacian by simply omitting the boundary integral terms.

\section{Finite Element Approximation of the Fractional Poisson Equation}
\label{sec:finite-elem-appr}

The \emph{fractional Poisson problem}
\begin{align*}
  \begin{aligned}
    \left(-\Delta\right)^{s}u &= f &&\text{in \(\Omega\),}\\
    u&=0&&\text{in \(\Omega^{c}\)}
  \end{aligned}
\end{align*}
takes the variational form
\begin{align}
  \text{Find } u\in \widetilde{H}^{s}\left(\Omega\right) : \quad a\left(u,v\right)=\left\langle f,v\right\rangle \quad \forall v\in \widetilde{H}^{s}\left(\Omega\right). \label{eq:fracPoisson}
\end{align}
Henceforth, let \(\Omega\) be a polygon, and let \(\mathcal{P}_{h}\) be a family of shape-regular and globally quasi-uniform triangulations of \(\Omega\), and \(\mathcal{P}_{h,\partial}\) the induced boundary meshes \cite{ErnGuermond2004_TheoryPracticeFiniteElements}.
Let \(\mathcal{N}_{h}\) be the set of vertices of \(\mathcal{P}_{h}\) and \(h_{K}\) be the diameter of the element \(K\in\mathcal{P}_{h}\), and \(h_{e}\) the diameter of \(e\in\mathcal{P}_{h,\partial}\).
Moreover, let \(h:=\max_{K\in\mathcal{P}_{h}}h_{K}\).
Let \(\phi_{i}\) be the usual piecewise linear basis function associated with a node \(\vec{z}_{i}\in\mathcal{N}_{h}\), satisfying \(\phi_{i}\left(\vec{z}_{j}\right)=\delta_{ij}\) for \(\vec{z}_{j}\in\mathcal{N}_{h}\), and let \(X_{h}:=\operatorname{span}\left\{\phi_{i}\mid \vec{z}_{i}\in\mathcal{N}_{h}\right\}\).
The finite element subspace \(V_{h}\subset \widetilde{H}^{s}\left(\Omega\right)\) is given by \(V_{h}=X_{h}\) when \(s<1/2\) and by
\begin{align*}
  V_{h} = \left\{v_{h}\in X_{h}\mid v_{h}=0 \text{ on }\partial\Omega\right\} = \operatorname{span}\left\{\phi_{i}\mid \vec{z}_{i}\not\in \partial\Omega\right\}
\end{align*}
when \(s\geq 1/2\).
The corresponding set of degrees of freedom \(\mathcal{I}_{h}\) for \(V_{h}\) is given by \(\mathcal{I}_{h}=\mathcal{N}_{h}\) when \(s<1/2\) and otherwise consists of nodes in the interior of \(\Omega\).
In both cases we denote the cardinality of \(\mathcal{I}_{h}\) by \(n\).
The set of degrees of freedom on an element \(K\in\mathcal{P}_{h}\) is denoted by \(\mathcal{I}_{K}\).

The stiffness matrix associated with the fractional Laplacian is defined to be \(\boldsymbol{A}^{s}=\left\{a\left(\phi_{i}, \phi_{j}\right)\right\}_{i,j}\), where
\begin{align*}
  a\left(\phi_{i},\phi_{j}\right)
  &= \frac{C(d,s)}{2} \int_{\Omega} \; d \vec{x} \int_{\Omega} \; d \vec{y} \frac{\left(\phi_{i}\left(\vec{x}\right)-\phi_{i}\left(\vec{y}\right)\right)\left(\phi_{j}\left(\vec{x}\right)-\phi_{j}\left(\vec{y}\right)\right)}{\absAinsworthGlusa{\vec{x}-\vec{y}}^{d+2s}} \\
  &\quad+ \frac{C(d,s)}{2s} \int_{\Omega} \; d \vec{x} \int_{\partial\Omega} \; d \vec{y} \frac{\phi_{i}\left(\vec{x}\right) \phi_{j}\left(\vec{x}\right) ~ \vec{n}_{\vec{y}}\cdot\left(\vec{x}-\vec{y}\right)}{\absAinsworthGlusa{\vec{x}-\vec{y}}^{d+2s}}.
\end{align*}
The existence of a unique solution to the fractional Poisson problem \cref{eq:fracPoisson} and its finite element approximation follows from the Lax-Milgram Lemma.
The rate of convergence of the finite element approximation is given by the following theorem:
\begin{theorem}[\cite{AcostaBorthagaray2015_FractionalLaplaceEquation}]\label{thm:Hsconv}
  If the family of triangulations \(\mathcal{P}_{h}\) is shape regular and globally quasi-uniform, and \(u\in H^{\ell}\left(\Omega\right)\), for \(0<s<\ell<1\) or \(1/2<s<1\) and \(1<\ell<2\), then
  \begin{align}
    \normAinsworthGlusa{u-u_{h}}_{\widetilde{H}^{s}\left(\Omega\right)}
    &\leq C\left(s,d\right) h^{\ell-s}\absAinsworthGlusa{u}_{H^{\ell}\left(\Omega\right)}.\label{eq:16}
  \end{align}
  In particular, by applying regularity estimates for \(u\) in terms of the data \(f\), the solution satisfies
  \begin{align*}
    \normAinsworthGlusa{u-u_{h}}_{\widetilde{H}^{s}\left(\Omega\right)} & \leq
                         \begin{cases}
                           C\left(s\right) h^{1/2}\absAinsworthGlusa{\log h} \normAinsworthGlusa{f}_{C^{1/2-s}\left(\Omega\right)} & \text{if } 0< s<1/2, \\
                           C h^{1/2}\absAinsworthGlusa{\log h} \normAinsworthGlusa{f}_{L^{\infty}\left(\Omega\right)} & \text{if } s=1/2, \\
                           \frac{C(s,\beta)}{2s-1} h^{1/2}\sqrt{\absAinsworthGlusa{\log h}} \normAinsworthGlusa{f}_{C^{\beta}\left(\Omega\right)} & \text{if } 1/2<s<1, \beta>0 \\
                         \end{cases}
  \end{align*}
\end{theorem}

Moreover, using a standard Aubin-Nitsche argument \cite[Lemma 2.31]{ErnGuermond2004_TheoryPracticeFiniteElements} gives estimates in \(L^{2}\left(\Omega\right)\):
\begin{theorem}[\cite{BorthagarayDelEtAl2016_FiniteElementApproximationFractionalEigenvalueProblem}]\label{thm:L2conv}
  If the family of triangulations \(\mathcal{P}_{h}\) is shape regular and globally quasi-uniform, and, for \(\epsilon>0\), \(u\in H^{s+1/2-\epsilon}\left(\Omega\right)\), then
  \begin{align*}
    \normAinsworthGlusa{u-u_{h}}_{L^{2}}\leq
    \begin{cases}
      C(s,\epsilon)h^{1/2+s-\epsilon}\absAinsworthGlusa{u}_{H^{s+1/2-\epsilon}\left(\Omega\right)} & \text{if } 0< s< 1/2, \\
      C(s,\epsilon)h^{1-2\epsilon}\absAinsworthGlusa{u}_{H^{s+1/2-\epsilon}\left(\Omega\right)} & \text{if } 1/2\leq s <1.
    \end{cases}
  \end{align*}
\end{theorem}

When \(s=1\) classical results \cite[Theorem 3.16 and Theorem 3.18]{ErnGuermond2004_TheoryPracticeFiniteElements} show that if \(u\in H^{\ell}\left(\Omega\right)\), \(1<\ell\leq2\),
\begin{align*}
  \normAinsworthGlusa{u-u_{h}}_{H^{1}_{0}\left(\Omega\right)}\leq C h^{\ell-1}\absAinsworthGlusa{u}_{H^{\ell}\left(\Omega\right)},\\
  \normAinsworthGlusa{u-u_{h}}_{L^{2}\left(\Omega\right)}\leq C h^{\ell}\absAinsworthGlusa{u}_{H^{\ell}\left(\Omega\right)},
\end{align*}
so that \eqref{eq:16} can be seen as a generalisation to the case \(s\in(0,1)\).
For \(s=1\), \(u\in H^{2}\left(\Omega\right)\) if the domain is of class \(C^{2}\) or a convex polygon and if \(f\in L^{2}\left(\Omega\right)\) \cite[Theorems 3.10 and 3.12]{ErnGuermond2004_TheoryPracticeFiniteElements}.
However, when \(s\in(0,1)\), higher order regularity of the solution is not guaranteed under such conditions.

For example, consider the problem
\begin{align*}
  \begin{aligned}
    \left(-\Delta\right)^{s}u^{s}(\vec{x}) &=1 && \text{in } \Omega=\left\{\vec{x}\in\mathbb{R}^{2}\mid \absAinsworthGlusa{\vec{x}}< 1 \right\},\\
    u^{s}\left(\vec{x}\right)&=0 && \text{in } \Omega^{c},
  \end{aligned}
\end{align*}
with analytic solution \cite{Getoor1961_FirstPassageTimesSymmetric}
\begin{align*}
  u^{s}\left(\vec{x}\right) := \frac{2^{-2s}}{\Gamma\left(1+s\right)^{2}} \left(1-\absAinsworthGlusa{\vec{x}}^{2}\right)^{s}.
\end{align*}
Although the domain is \(C^{\infty}\) and the right-hand side is smooth, \(u^{s}\) is only in \(H^{s+1/2-\epsilon}\left(\Omega\right)\) for any \(\epsilon>0\).
Sample solutions for \(s\in\left\{0.25, 0.75\right\}\) are shown in \Cref{fig:solutions}.
\begin{figure}
  \centering
  % created using dodo.py
  \includegraphics{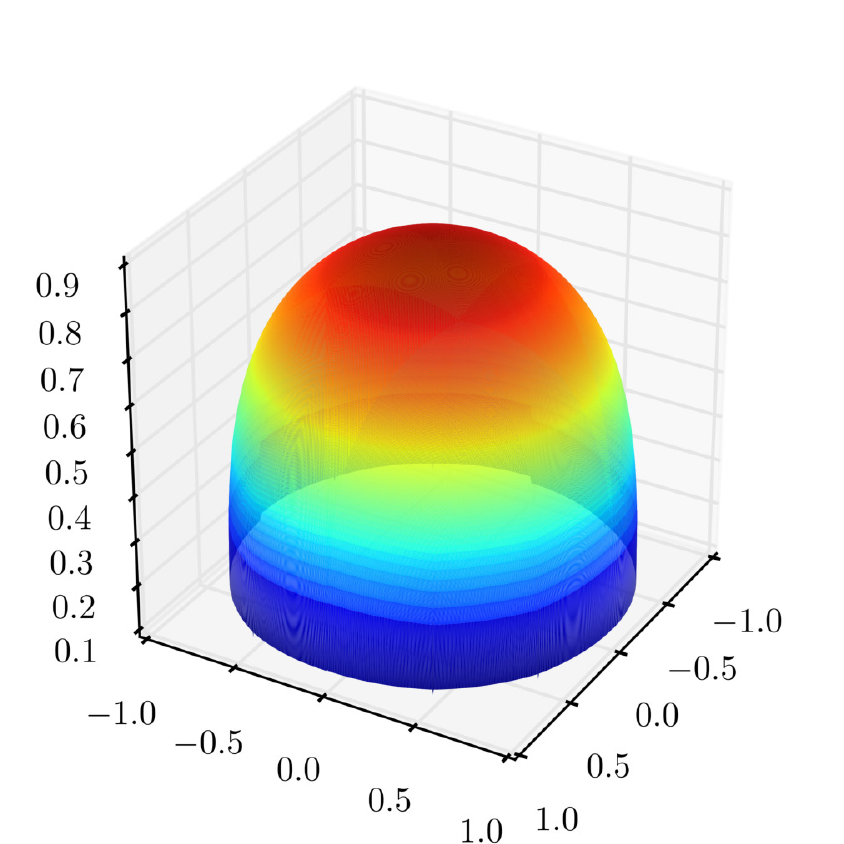}
  \includegraphics{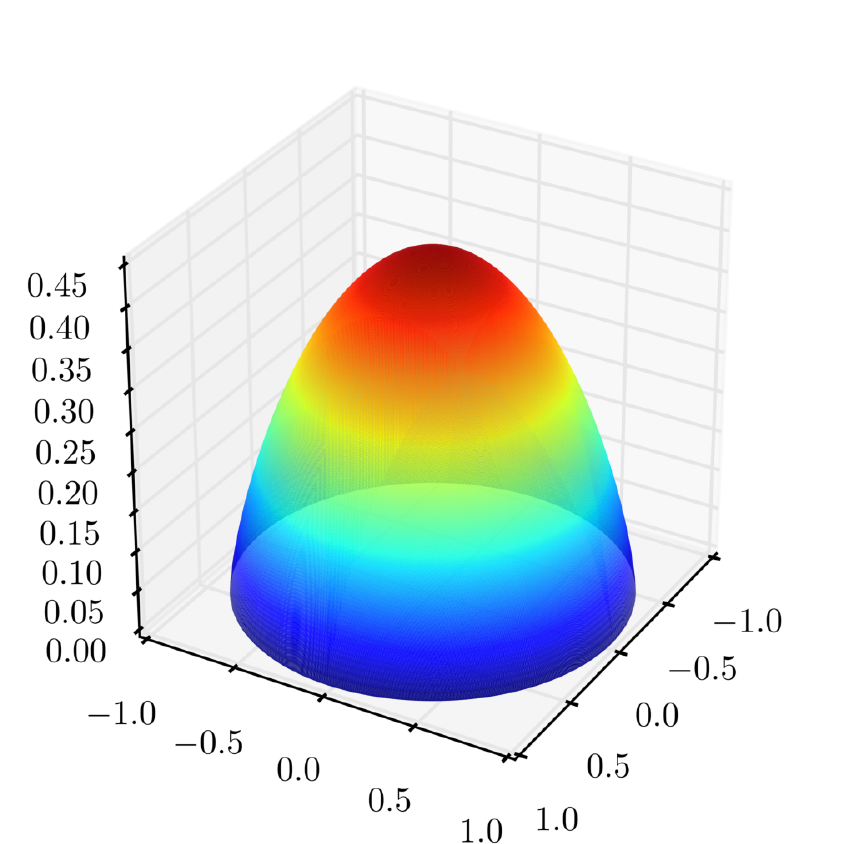}
  \caption{Solutions \(u^{s}\) to the fractional Poisson equation with constant right-hand side for \(s=0.25\) (\emph{top}) and \(s=0.75\) (\emph{bottom}).}
  \label{fig:solutions}
\end{figure}

\section{Computation of Entries of the Stiffness Matrix}
\label{sec:comp-entr-stiffn}

The computation of entries of the stiffness matrix \(\boldsymbol{A}^{s}\) in the case of the usual Laplacian (\(s=1\)) is straightforward.
However, for \(s\in(0,1)\), the bilinear form contains factors \(\absAinsworthGlusa{\vec{x}-\vec{y}}^{-d-2s}\) which means that simple closed forms for the entries are no longer available and suitable quadrature rules therefore must be identified.
Moreover, the presence of a repeated integral over \(\Omega\) (as opposed to an integral over just \(\Omega\) in the case \(s=1\)) means that the matrix needs to be assembled in a double loop over the elements of the mesh so that the computational cost is potentially much larger than in the integer \(s=1\) case.
Additionally, every degree of freedom is coupled to all other degrees of freedom and the stiffness matrix is therefore dense.

\subsection{Reduction to Smooth Integrals}
\label{sec:reduct-smooth-integr}

In order to compute the entries of \(\boldsymbol{A}^{s}=\left\{a\left(\phi_{i},\phi_{j}\right)\right\}_{ij}\) we decompose the expression for the entries into contributions from elements \(K,\tilde{K}\in\mathcal{P}_{h}\) and external edges \(e\in\mathcal{P}_{h,\partial}\):
\begin{align*}
  a(\phi_{i},\phi_{i}) &= \sum_{K}\sum_{\tilde{K}} a^{K\times\tilde{K}}(\phi_{i},\phi_{j}) + \sum_{K}\sum_{e} a^{K\times e}(\phi_{i},\phi_{j}),
\end{align*}
where the contributions \(a^{K\times\tilde{K}}\) and \(a^{K\times e}\) are given by:
\begin{align}
  a^{K\times\tilde{K}}(\phi_{i},\phi_{j})
  & = \frac{C(d,s)}{2} \int_{K} \; d \vec{x} \int_{\tilde{K}} \; d \vec{y} \frac{\left(\phi_{i}(\vec{x})-\phi_{i}(\vec{y})\right)\left(\phi_{j}(\vec{x})-\phi_{j}(\vec{y})\right)}{\absAinsworthGlusa{\vec{x}-\vec{y}}^{d+2s}}, \label{eq:6}\\
    a^{K\times e}(\phi_{i},\phi_{j})
  &= \frac{C(d,s)}{2s} \int_{K} \; d \vec{x} \int_{e} \; d \vec{y} \frac{\phi_{i}\left(\vec{x}\right) \phi_{j}\left(\vec{x}\right) ~ \vec{n}_{e}\cdot\left(\vec{x}-\vec{y}\right)}{\absAinsworthGlusa{\vec{x}-\vec{y}}^{d+2s}}.\label{eq:7}
\end{align}
Although the following approach holds for arbitrary spatial dimension \(d\), we restrict ourselves to \(d=2\) dimensions.
In evaluating the contributions \(a^{K\times\tilde{K}}\) over element pairs \(K\times\tilde{K}\), several cases need to be distinguished:
\begin{enumerate}
\item \(K\) and \(\tilde{K}\) have empty intersection,
\item \(K\) and \(\tilde{K}\) are identical,
\item \(K\) and \(\tilde{K}\) share an edge,
\item \(K\) and \(\tilde{K}\) share a vertex.
\end{enumerate}
These cases are illustrated in \Cref{fig:elementConfigs}.
\begin{figure}
  \centering
  \begin{tikzpicture}
    \coordinate (node0) at (0,0);
    \coordinate (node1) at (1,0);
    \coordinate (node2) at (0.5,0.866);
    \coordinate (node3) at (1.5,0.866);
    \coordinate (node4) at (2,0);
    \coordinate (node5) at (0.5,-0.866);
    \coordinate (node6) at (1.5,-0.866);
    \coordinate (node7) at (2.5,0.866);
    \coordinate (node8) at (2.5,-0.866);
    \coordinate (node9) at (3,0);
    \coordinate (node10) at (3.5,0.866);
    \coordinate (node11) at (3.5,-0.866);
    \coordinate (node12) at (4,0);
    \coordinate (node13) at (4.5,0.866);
    \coordinate (node14) at (4.5,-0.866);
    \coordinate (node15) at (5,0);
    \coordinate (node16) at (5.5,0.866);
    \coordinate (node17) at (5.5,-0.866);
    \coordinate (node18) at (6,0);

    \draw[fill=red] (node0) -- (node1) -- (node2) -- cycle;
    \draw (node1) -- (node3) -- (node2) -- cycle;
    \draw (node1) -- (node4) -- (node3) -- cycle;
    \draw (node0) -- (node1) -- (node5) -- cycle;
    \draw[fill=yellow] (node1) -- (node4) -- (node6) -- cycle;
    \draw[fill=yellow] (node1) -- (node5) -- (node6) -- cycle;

    \draw (node4) -- (node3) -- (node7) -- cycle;
    \draw (node4) -- (node6) -- (node8) -- cycle;
    \draw[fill=blue] (node4) -- (node9) -- (node7) -- cycle;
    \draw (node4) -- (node9) -- (node8) -- cycle;

    \draw (node9) -- (node7) -- (node10) -- cycle;
    \draw (node9)[fill=blue] -- (node8) -- (node11) -- cycle;
    \draw (node9) -- (node12) -- (node10) -- cycle;
    \draw (node9) -- (node12) -- (node11) -- cycle;

    \draw[fill=green] (node12) -- (node10) -- (node13) -- cycle;
    \draw (node12) -- (node11) -- (node14) -- cycle;
    \draw (node12) -- (node15) -- (node13) -- cycle;
    \draw (node12) -- (node15) -- (node14) -- cycle;

    \draw (node15) -- (node13) -- (node16) -- cycle;
    \draw (node15) -- (node14) -- (node17) -- cycle;
    \draw[fill=green] (node15) -- (node18) -- (node16) -- cycle;
    \draw (node15) -- (node18) -- (node17) -- cycle;
  \end{tikzpicture}
  \caption{Element pairs that are treated separately. We distinguish element pairs of identical elements (\emph{red}), element pairs with common edge (\emph{yellow}), with common vertex (\emph{blue}) and separated elements (\emph{green}).}\label{fig:elementConfigs}
\end{figure}
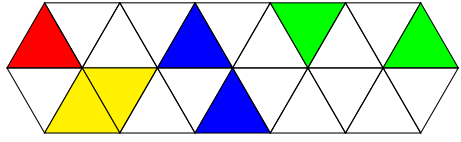
In case 1, where the elements do not touch, the Stroud conical quadrature rule \cite{Stroud1971_ApproximateCalculationMultipleIntegrals} (or any other suitable Gauss rule on simplices) of \emph{sufficiently high order} can be used to approximate the integrals.
More details as to what constitutes a sufficiently high order are given in \Cref{sec:quadrature-errors}.

Special care has to be taken in the remaining cases 2-4, in which the elements are touching, owing to the presence of a singularity in the integrand.
Fortunately, the singularity is removable and can, as pointed out in \cite{AcostaBorthagaray2015_FractionalLaplaceEquation}, be treated using standard techniques from the boundary element literature \cite{SauterSchwab2010_BoundaryElementMethods}.
More specifically, we write the integral as a sum of integrals over sub-simplices.
Each sub-simplex is then mapped onto the hyper-cube \([0,1]^{4}\) using the Duffy transformation \cite{Duffy1982_QuadratureOverPyramidOr}.
The advantage of pursuing this approach is that the singularity arising from the degenerate nature of the Duffy transformation offsets the singularity present in the integrals.
For example, we obtain the following expressions
\begin{multline}
  a^{K\times\tilde{K}}(\phi_{i},\phi_{j})=
  \frac{C(2,s)}{2} \frac{\absAinsworthGlusa{K}}{\absAinsworthGlusa{\hat{K}}} \frac{\absAinsworthGlusa{\tilde{K}}}{\absAinsworthGlusa{\hat{K}}} \sum_{\ell=1}^{L_{c}} \int_{[0,1]^{4}} \; d \vec{\eta} ~ \Bar J^{(\ell,c)}\frac{\Bar \psi_{k(i)}^{(\ell,c)}\left(\vec{\eta}\right) \Bar \psi_{k(j)}^{(\ell,c)}\left(\vec{\eta}\right)}{\absAinsworthGlusa{\sum_{k=0}^{6-c}\Bar\psi_{k}^{(\ell,c)}\left(\vec{\eta}\right) \vec{x}_{k}}^{2+2s}}, \label{eq:8}
\end{multline}
and
\begin{multline}
  a^{K\times e}(\phi_{i},\phi_{j}) =\\
  \frac{C(2,s)}{2s} \frac{\absAinsworthGlusa{K}}{\absAinsworthGlusa{\hat{K}}} \frac{\absAinsworthGlusa{e}}{\absAinsworthGlusa{\hat{e}}} \sum_{\ell=1}^{L_{c}} \int_{[0,1]^{3}} \; d \vec{\eta} ~ \Bar J^{(\ell,c)}
  \frac{\phi_{k(i)}^{(\ell,c)}\left(\vec{\eta}\right) \phi_{k(j)}^{(\ell,c)}\left(\vec{\eta}\right) ~ \sum_{k=0}^{5-c}\Bar\psi_{k}^{(\ell,c)}\left(\vec{\eta}\right) \vec{n}_{e}\cdot\vec{x}_{k}}{\absAinsworthGlusa{\sum_{k=0}^{5-c}\Bar\psi_{k}^{(\ell,c)}\left(\vec{\eta}\right) \vec{x}_{k}}^{2+2s}}\label{eq:9}
\end{multline}
in which the singularity \(\absAinsworthGlusa{\vec{x}-\vec{y}}^{-d-2s}\) is no longer present.
The derivations of the terms involved can be found in \cite{SauterSchwab2010_BoundaryElementMethods,AcostaBersetcheEtAl2016_ShortFeImplementation2d} and, for completeness, are summarised in the appendix, along with the notations used in equations \cref{eq:8,eq:9}.
Removing the singularity means that the integrals in \cref{eq:8,eq:9} are amenable to approximation using standard Gaussian quadrature rules of \emph{sufficiently high  order} as discussed in \Cref{sec:quadrature-errors}.
The same idea is applicable in any number of space dimensions.

\subsection{Determining the Order of the Quadrature Rules}
\label{sec:quadrature-errors}

The foregoing considerations show that the evaluation of the entries of the stiffness matrix boils down to the evaluation of integrals with smooth integrands, i.e. expressions \cref{eq:6,eq:7} for case 1 and expressions \cref{eq:8,eq:9} for case 2-4.
As mentioned earlier, it is necessary to use a \emph{sufficiently high order} quadrature rule to approximate these integrals.
We now turn to the question of how high is sufficient.

The arguments used to prove the ensuing estimates follow a pattern similar to the proofs of Theorems 5.3.29, 5.3.23 and 5.3.24 in \cite{SauterSchwab2010_BoundaryElementMethods}.
The main difference from \cite{SauterSchwab2010_BoundaryElementMethods} is the presence of the boundary integral term.
More details on the development of this type of quadrature rules in the context of boundary element methods can be found in the work of Erichsen and Sauter \cite{ErichsenSauter1998_EfficientAutomaticQuadrature3}.

\begin{restatable}{theorem}{consistencyErrorQuadrature}
  For \(d=2\), let \(\mathcal{I}_{K}\) index the degrees of freedom on \(K\in\mathcal{P}_{h}\), and define \(\mathcal{I}_{K\times\tilde{K}}:=\mathcal{I}_{K}\cup\mathcal{I}_{\tilde{K}}\).
  Let \(k_{T}\) (respectively \(k_{T,\partial}\)) be the quadrature order used for touching pairs \(K\times\tilde{K}\) (respectively \(K\times e\)), and let \(k_{NT}\left(K,\tilde{K}\right)\) (respectively \(k_{NT,\partial}\left(K,e\right)\)) be the quadrature order used for pairs that have empty intersection.
  Denote the resulting approximation to the bilinear form \(a\left(\cdot,\cdot\right)\) by \(a_{Q}\left(\cdot,\cdot\right)\).
  Then the  consistency error due to quadrature is bounded by
  \begin{align*}
    \absAinsworthGlusa{a(u,v)-a_{Q}(u,v)}&\leq C \left(E_{T} + E_{NT} + E_{T,\partial} + E_{NT,\partial}\right) \normAinsworthGlusa{u}_{L^{2}\left(\Omega\right)} \normAinsworthGlusa{v}_{L^{2}\left(\Omega\right)} \quad \forall u,v \in V_{h},
  \end{align*}
  where the errors are given by
  \begin{align*}
    E_{T}&=h^{-2-2s}\rho_{1}^{-2k_{T}},\\
    E_{NT} &= \max_{K,\tilde{K}\in\mathcal{P}_{h}, \overline{K}\cap\overline{\tilde{K}}=\emptyset} h^{-2}d_{K,\tilde{K}}^{-2s}\left(\rho_{2}\frac{d_{K,\tilde{K}}}{h}\right)^{-2k_{NT}\left(K,\tilde{K}\right)}, \\
    E_{T,\partial} &= h^{-1-2s}\rho_{3}^{-2k_{T,\partial}},\\
    E_{NT,\partial} &= \max_{K\in\mathcal{P}_{h}, e\in\mathcal{P}_{h,\partial}, \overline{K}\cap\overline{e}=\emptyset} h^{-1}d_{K,e}^{-2s}\left(\rho_{4}\frac{d_{K,e}}{h}\right)^{-2k_{NT,\partial}\left(K,e\right)},
  \end{align*}
  \(d_{K,\tilde{K}}:=\inf_{\vec{x}\in K, \vec{y}\in\tilde{K}}\absAinsworthGlusa{\vec{x}-\vec{y}}\), \(d_{K,e}:=\inf_{\vec{x}\in K, \vec{y}\in e}\absAinsworthGlusa{\vec{x}-\vec{y}}\), and \(\rho_{j}>1\), \(j=1,2,3,4\), are constants.
\end{restatable}

The proof of the Theorem is deferred to the appendix.

The impact of the use of quadrature rules on the accuracy of the resulting finite element approximation can be quantified using Strang's first lemma \cite[Lemma 2.27]{ErnGuermond2004_TheoryPracticeFiniteElements}:
\begin{align*}
  \normAinsworthGlusa{u-u_{h}}_{\widetilde{H}^{s}\left(\Omega\right)}
  &\leq C \inf_{v_{h}\in V_{h}}\left[\normAinsworthGlusa{u-v_{h}}_{\widetilde{H}^{s}\left(\Omega\right)} + \sup_{w_{h}\in V_{h}}\frac{\absAinsworthGlusa{a(v_{h},w_{h})-a_{Q}(v_{h},w_{h})}}{\normAinsworthGlusa{w_{h}}_{\widetilde{H}^{s}\left(\Omega\right)}}\right] \\
  &\leq C \inf_{v_{h}\in V_{h}}\left[ \normAinsworthGlusa{u-v_{h}}_{\widetilde{H}^{s}\left(\Omega\right)} \right. \\
  &\quad\left.+ \left(E_{T} + E_{NT} + E_{T,\partial} + E_{NT,\partial}\right) \normAinsworthGlusa{v_{h}}_{L^{2}(\Omega)} \sup_{w_{h}\in V_{h}}\frac{\normAinsworthGlusa{w_{h}}_{L^{2}(\Omega)}}{\normAinsworthGlusa{w_{h}}_{\widetilde{H}^{s}\left(\Omega\right)}}\right] \\
  &\leq C \inf_{v_{h}\in V_{h}}\left[ \normAinsworthGlusa{u-v_{h}}_{\widetilde{H}^{s}\left(\Omega\right)} + \left(E_{T} + E_{NT} + E_{T,\partial} + E_{NT,\partial}\right) \normAinsworthGlusa{v_{h}}_{L^{2}(\Omega)}\right],
\end{align*}
where we used the Poincare inequality \(\normAinsworthGlusa{w_{h}}_{L^{2}(\Omega)}\leq C \normAinsworthGlusa{w_{h}}_{\widetilde{H}^{s}\left(\Omega\right)}\) in the last step.
We then use the Scott-Zhang interpolation operator \(\Pi_{h}\) \cite{Ciarlet2013_AnalysisScott,ScottZhang1990_FiniteElementInterpolationNonsmooth} and the estimate
\begin{align*}
  \normAinsworthGlusa{u-\Pi_{h}u}_{\widetilde{H}^{s}\left(\Omega\right)}&\leq Ch^{\ell-s}\absAinsworthGlusa{u}_{H^{\ell}(\Omega)},
\end{align*}
used in the proof of \Cref{thm:Hsconv} to bound the first term on the right-hand side:
\begin{align*}
  \normAinsworthGlusa{u-u_{h}}_{\widetilde{H}^{s}\left(\Omega\right)}
  &\leq C \left[ h^{\ell-s}\absAinsworthGlusa{u}_{H^{\ell}\left(\Omega\right)} + \left(E_{T} + E_{NT} + E_{T,\partial} + E_{NT,\partial}\right) \normAinsworthGlusa{\Pi_{h}u}_{L^{2}(\Omega)}\right].
\end{align*}
We choose the quadrature rules in such a way that the remaining terms on the right-hand side are also of order \(\mathcal{O}\left(h^{\ell-s}\right)\), i.e.
\begin{align}
  k_{T} &\geq \frac{\left(\ell-s+2+2s\right)}{2\log(\rho_{1})}\absAinsworthGlusa{\log h}-C,\label{eq:17}\\
  k_{NT}\left(K,\tilde{K}\right) &\geq \frac{\left((\ell-s)/2+1+s\right)\absAinsworthGlusa{\log h} - s\log\frac{d_{K,\tilde{K}}}{h} - C}{\log\frac{d_{K,\tilde{K}}}{h} + \log(\rho_{2})},\label{eq:18}\\
  k_{T,\partial} &\geq \frac{\left(\ell-s+1+2s\right)}{2\log(\rho_{3})}\absAinsworthGlusa{\log h}-C,\label{eq:19}\\
  k_{NT,\partial}\left(K,e\right) &\geq \frac{\left((\ell-s)/2+1/2+s\right)\absAinsworthGlusa{\log h} - s\log\frac{d_{K,e}}{h} - C}{\log\frac{d_{K,e}}{h} + \log(\rho_{4})}.\label{eq:20}
\end{align}
In particular, if the pair \(K\times\tilde{K}\) (respectively \(K\times e\)) is well separated, so that \(d_{K,\tilde{K}}\sim 1\) (\(d_{K,e}\sim 1\)), then
\begin{align*}
  k_{NT}\left(K,\tilde{K}\right)&\geq (\ell-s)/2+1,\\
  k_{NT,\partial}\left(K,e\right)&\geq (\ell-s)/2+1/2
\end{align*}
is sufficient.

In practice, the quadrature order for non-touching element pairs can be chosen depending on \(d_{K,\tilde{K}}\) using \cref{eq:18,eq:20}, or an appropriate choice of cutoff distance \(D\) can be determined so that element pairs with \(d_{K,\tilde{K}}<D\) are approximated using a quadrature rule with \(\mathcal{O}\left(\absAinsworthGlusa{\log h}\right)\) nodes, and pairs with \(d_{K,\tilde{K}}\geq D\) are computed using a constant number of nodes.

It transpires from the expressions derived in the appendix and the fact that \(n\sim h^{-2}\) that the complexity to calculate the contributions by a single pair of elements \(K\) and \(\tilde{K}\) scales like
\begin{itemize}
\item \(\log n\) if the elements coincide,
\item \(\left(\log n\right)^{2}\) if the elements share only an edge,
\item \(\left(\log n\right)^{3}\) if the elements share only a vertex,
\item \(\left(\log n\right)^{4}\) if the elements have empty intersection, but are ``near neighbours'', and
\item \(C\) if the elements are well separated.
\end{itemize}
Since \(n\sim \absAinsworthGlusa{\mathcal{P}_{h}}\), we cannot expect a straightforward assembly of the stiffness matrix to scale better than \(\mathcal{O}\left(n^{2}\right)\).
Similarly, its memory requirement is \(n^{2}\), and a single matrix-vector product has complexity \(\mathcal{O}\left(n^{2}\right)\), which severely limits the size of problems that can be considered.

\section{Solving the Linear Systems}
\label{sec:solut-syst-involv}

The fractional Poisson equation leads to the linear algebraic system
\begin{align}
  \boldsymbol{A}^{s}\vec{u}=\vec{b},\label{eq:4}
\end{align}
whereas time-dependent problems (using implicit integration schemes) lead to systems of the form
\begin{align}
    \left(\boldsymbol{M}+\Delta t\boldsymbol{A}^{s}\right)\vec{u}=\vec{b},\label{eq:3}
\end{align}
where \(\Delta t\) is the time-step size.
In typical examples, the time-step will be chosen so that the orders of convergence in both spatial and temporal discretisation errors are balanced.

In both cases, the matrices are dense and the condition number of \(\boldsymbol{A}^{s}\) grows as the mesh is refined (\(h\rightarrow0\)).
The cost of using a direct solver is prohibitively expensive, growing as \(\mathcal{O}\left(n^{3}\right)\).
An alternative is to use an iterative solver such as the conjugate gradient method but the rate of convergence will depend on the condition number.
The following result quantifies how the condition number of \(\boldsymbol{A}^{s}\) depends on the fractional order \(s\) and the mesh size \(h\):
\begin{theorem}[\cite{AinsworthMcleanEtAl1999_ConditioningBoundaryElementEquations}]\label{thm:condFracLapl}
  For \(s<d/2\), and a family of shape regular and globally quasi-uniform triangulations \(\mathcal{P}_{h}\) with maximal element size \(h\), the spectrum of the stiffness matrix satisfies
  \begin{align*}
    ch^{d}\boldsymbol{I} \leq \boldsymbol{A}^{s} \leq Ch^{d-2s}\boldsymbol{I},
  \end{align*}
  and hence the condition number of the stiffness matrix satisfies
  \begin{align*}
    \kappa\left(\boldsymbol{A}^{s}\right) &= C h^{-2s}.
  \end{align*}
\end{theorem}

The exponent of the growth of the condition number depends on the fractional order \(s\).
For small \(s\), the matrix is better conditioned, similarly to the mass matrix in the case of integer order operators.
As \(s\rightarrow 1\), the growth of the condition number approaches \(\mathcal{O}\left(h^{-2}\right)\), as for the usual Laplacian.
Consequently, just as the conjugate gradient method fails to be efficient for the solution of equations arising from the discretisation of the Laplacian, CG becomes increasingly uncompetitive for the solution of equations arising from the fractional Laplacian.

In the integer order case, multigrid iterations have been used with great success for solving systems involving both the mass matrix and the stiffness matrix that arises from the discretisation of the regular Laplacian.
It is therefore to be expected that the same will remain true for systems arising from the fractional Laplacian.
In practice, a single multigrid iteration is much more expensive than a single iteration of conjugate gradient.
The advantage of multigrid is, however, that the number of iterations is essentially independent of the number of unknowns \(n\).
Consequently, while the performance of CG degenerates as \(n\) increases, this will not be the case with multigrid making it attractive as a solver for the fractional Poisson problem.

Turning to the systems that arise from the discretisation of time-dependent problems, we first observe that an explicit scheme will lead to CFL conditions on the time-step size of the form \(\Delta t \leq C h^{2s}\).
On the other hand, for implicit time-stepping, the following theorem shows that if the time-step \(\Delta t=\mathcal{O}\left(h^{2s}\right)\), we can expect the conjugate gradient method to converge rapidly, at a rate which does not degenerate as \(n\) increases, in contrast with what is observed for steady problems:
\begin{lemma}\label{lem:condFracLaplTime}
  For a shape regular and globally quasi-uniform family of triangulations \(\mathcal{P}_{h}\) and time-step \(\Delta t \leq 1\),
  \begin{align*}
    \kappa\left(\boldsymbol{M}+\Delta t\boldsymbol{A}^{s}\right)
    &\leq C \left(1+\frac{\Delta t}{h^{2s}}\right).
  \end{align*}
\end{lemma}
\begin{proof}~
  Since \(ch^{d}\boldsymbol{I}\leq \boldsymbol{M} \leq Ch^{d}\boldsymbol{I}\), this also permits us to deduce that
  \begin{align*}
    c\left(h^{d} +  \Delta t~ h^{d}\right)\boldsymbol{I} \leq \boldsymbol{M}+\Delta t \boldsymbol{A}^{s} \leq C\left(h^{d} + \Delta t~ h^{d-2s}\right)\boldsymbol{I}
  \end{align*}
  and so
  \begin{align*}
    \kappa\left(\boldsymbol{M}+\Delta t\boldsymbol{A}^{s}\right)
    &\leq  C  \left(1+\frac{\Delta t}{h^{2s}}\right).
  \end{align*}
\end{proof}
This shows that for a general time-step \(\Delta t\geq h^{2s}\), the number of iterations the conjugate gradient method will require for systems of the form \cref{eq:3} will grow as \(\sqrt{\Delta t/h^{2s}} \sim n^{s/d}\sqrt{\Delta t} \).
Consequently, if \(\Delta t\) is large compared to \(h^{2s}\), a multigrid solver outperforms conjugate gradient for the systems \cref{eq:3}, but if \(\Delta t\) is on the same order as \(h^{2s}\), conjugate gradient iterations will generally be more efficient than a multigrid method.

In this section we have concerned ourselves with the effect that the mesh and the fractional order have on the rate of convergence of iterative solvers.
This, of course, ignores the cost of carrying out the iteration in which a matrix-vector multiply must be computed at each step.
The complexity of both multigrid and conjugate gradient iterations depends on how efficiently the matrix-vector product \(\boldsymbol{A}^{s}\vec{x}\) can be computed.
By way of contrast, the mass matrix in \cref{eq:3} has \(\mathcal{O}\left(n\right)\) entries, so its matrix-vector product scales linearly in the number of unknowns.
Since all the basis functions \(\phi_{i}\) interact with one another, the matrix \(\boldsymbol{A}^{s}\) is dense and the associated matrix-vector product has complexity \(\mathcal{O}\left(n^{2}\right)\).
In the following section, we discuss a sparse approximation that will preserve the order of the approximation error of the fractional Laplacian, but display significantly better scaling in terms of both memory usage and operation counts for both assembly and matrix-vector product.

\section{Sparse Approximation of the Matrix}
\label{sec:cluster-method}

The presence of a factor \(\absAinsworthGlusa{\vec{x}-\vec{y}}^{-d-2s}\) in the integrand in the expression for the entries of the stiffness matrix means that the contribution of pairs of elements that are well separated is significantly smaller than the contribution arising from pairs of elements that are close to one another.
This suggests the use of the \emph{panel clustering method} \cite{HackbuschNowak1989_FastMatrixMultiplicationBoundary} from the boundary element literature, whereby such far field contributions are replaced by less expensive low-rank blocks rather than computing and storing all the individual entries from the original matrix.
Conversely, the near-field contributions are more significant but involve only local couplings and hence the cost of storing the individual entries is a practical proposition.
A full discussion of the panel clustering method is beyond the scope of the present work but can be found in \cite[Chapter 7]{SauterSchwab2010_BoundaryElementMethods}.
Here, we confine ourselves to stating only the necessary definitions and steps needed to describe our approach.

\begin{definition}[\cite{SauterSchwab2010_BoundaryElementMethods}]
  A \emph{cluster} is a union of one or more indices from the set of degrees of freedom \(\mathcal{I}\).
  The nodes of a hierarchical cluster tree \(\mathcal{T}\) are clusters.
  The set of all nodes is denoted by \(T\) and satisfies
  \begin{enumerate}
  \item \(\mathcal{I}\) is a node of \(\mathcal{T}\).
  \item The set of leaves \(\text{Leaves}(\mathcal{T})\subset T\) corresponds to the degrees of freedom \(i\in\mathcal{I}\) and is given by
    \begin{align*}
      \text{Leaves}(\mathcal{T}) := \left\{\left\{i\right\} : i\in\mathcal{I}\right\}.
    \end{align*}
  \item For every \(\sigma\in T\setminus \text{Leaves}\left(\mathcal{T}\right)\) there exists a minimal set \(\Sigma\left(\sigma\right)\) of nodes in \(T\setminus\left\{\sigma\right\}\) (i.e. of minimal cardinality) that satisfies
    \begin{align*}
      \sigma = \bigcup_{\tau\in\Sigma\left(\sigma\right)}\tau.
    \end{align*}
    The set \(\Sigma\left(\sigma\right)\) is called the sons of \(\sigma\).
    The edges of the cluster tree \(\mathcal{T}\) are the pairs of nodes \(\left(\sigma,\tau\right)\in T\times T\) such that \(\tau\in \Sigma\left(\sigma\right)\).
  \end{enumerate}
\end{definition}
An example of a cluster tree for a one-dimensional problem is given in \Cref{fig:clusterTree}.

\begin{figure}
  \centering
  % created using dodo.py
  \includegraphics{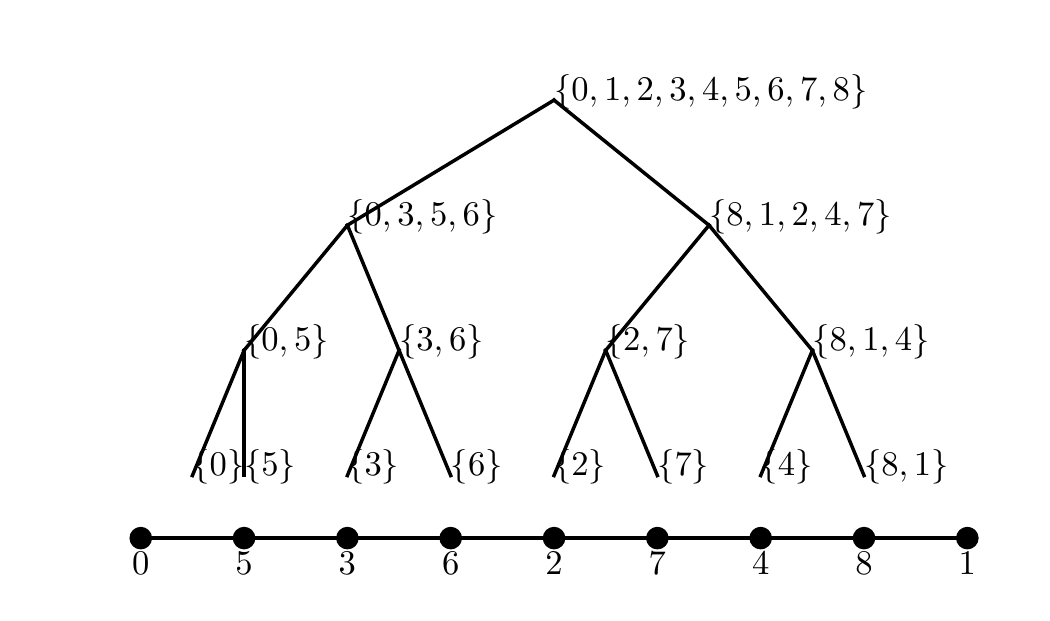}
  \caption{
    Cluster tree for a one dimensional problem.
    For each cluster, the associated degrees of freedom are shown.
    The mesh with its nodal degrees of freedom is plotted at the bottom.
  }
  \label{fig:clusterTree}
\end{figure}

\begin{definition}[\cite{SauterSchwab2010_BoundaryElementMethods}]
  The \emph{cluster box} \(Q_{\sigma}\) of a cluster \(\sigma\in T\) is the minimal hyper-cube which contains \(\bigcup_{i\in\sigma}\operatorname{supp} \phi_{i}\).
  The \emph{diameter} of a cluster is the diameter of its cluster box \(\operatorname{diam}\left(\sigma\right):= \sup_{\vec{x},\vec{y}\in Q_{\sigma}} \absAinsworthGlusa{\vec{x}-\vec{y}}\).
  The \emph{distance} of two clusters \(\sigma\) and \(\tau\) is \(\operatorname{dist}\left(\sigma,\tau\right):=\inf_{\vec{x}\in Q_{\sigma}, \vec{y}\in Q_{\tau}}\absAinsworthGlusa{\vec{x}-\vec{y}}\).
  The subspace \(V_{\sigma}\) of \(V_{h}\) is defined as \(V_{\sigma}:=\operatorname{span}\left\{\phi_{i}\mid i\in\sigma\right\}\).
\end{definition}

For given \(\eta>0\), a pair of clusters \(\left(\sigma,\tau\right)\) is called \emph{admissible}, if
\begin{align*}
  \eta\operatorname{dist}\left(\sigma,\tau\right)& \geq \max\left\{\operatorname{diam}\left(\sigma\right), \operatorname{diam}\left(\tau\right)\right\}.
\end{align*}
The admissible cluster pairs can be determined recursively.
Cluster pairs that are not admissible and have no admissible sons are part of the near field and are assembled into a sparse matrix.
The admissible cluster pairs for a one dimensional problem are shown in \Cref{fig:clusterPairs}.

\begin{figure}
  \centering
  % created using dodo.py
  \includegraphics{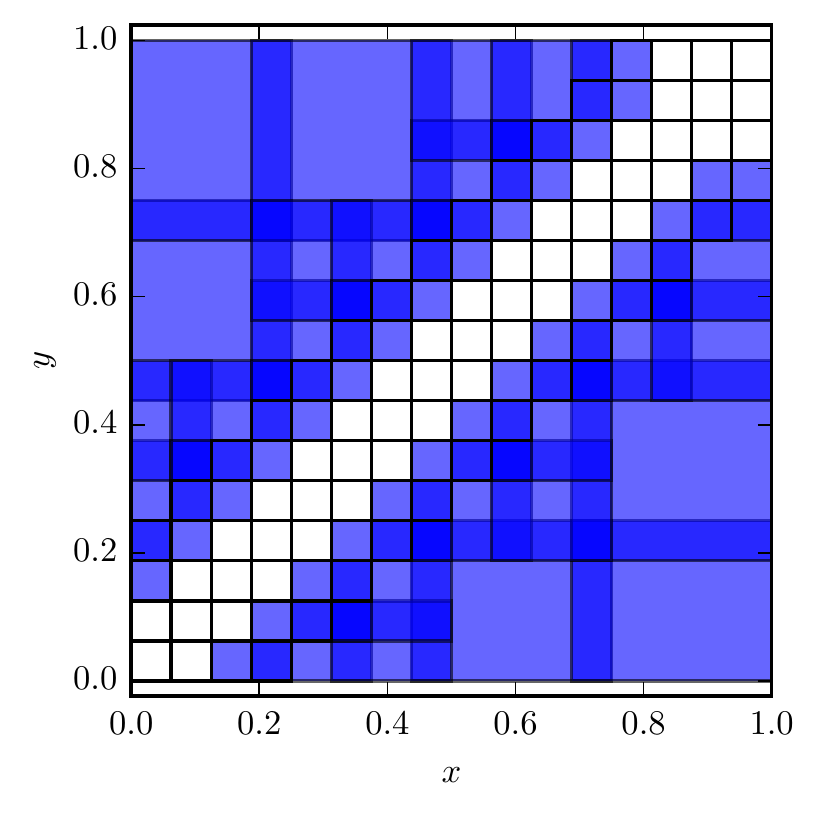}
  \caption{
    Cluster pairs for a one dimensional problem.
    The cluster boxes of the admissible cluster pairs are coloured in light blue, and their overlap in darker blue.
    The diagonal cluster pairs are not admissible and are not approximated, but assembled in full.
  }
  \label{fig:clusterPairs}
\end{figure}

For admissible pairs of clusters \(\sigma\) and \(\tau\) and any degrees of freedom \(i\in\sigma\) and \(j\in\tau\), the corresponding entry of the stiffness matrix is
\begin{align*}
  \left(\boldsymbol{A}^{s}\right)_{ij}
  &= a\left(\phi_{i},\phi_{j}\right) =-C\left(d,s\right)\int_{\Omega} \int_{\Omega} k\left(\vec{x},\vec{y}\right)\phi_{i}\left(\vec{x}\right) \phi_{j}\left(\vec{y}\right)
\end{align*}
with kernel \(k\left(\vec{x},\vec{y}\right)=\absAinsworthGlusa{\vec{x}-\vec{y}}^{-(d+2s)}\).
The kernel can be approximated on \(Q_{\sigma}\times Q_{\tau}\) using Chebyshev interpolation of order \(m\) in every spatial dimension by
\begin{align*}
  k_{m}\left(\vec{x},\vec{y}\right)&= \sum_{\alpha,\beta=1}^{m^{d}} k\left(\vec{\xi}_{\alpha}^{\sigma},\vec{\xi}_{\beta}^{\tau}\right) L_{\alpha}^{\sigma}\left(\vec{x}\right) L_{\beta}^{\tau}\left(\vec{y}\right).
\end{align*}
Here, \(\vec{\xi}_{\alpha}^{\sigma}\) are the tensor Chebyshev nodes on \(Q_{\sigma}\), and \(L_{\alpha}^{\sigma}\) are the associated Lagrange polynomials on the cluster box \(Q_{\sigma}\) with \(L_{\alpha}^{\sigma}\left(\vec{\xi}_{\beta}^{\sigma}\right)=\delta_{\alpha\beta}\).
This leads to the following approximation:
\begin{align*}
  \left(\boldsymbol{A}^{s}\right)_{ij}
  &\approx  -C\left(d,s\right) \sum_{\alpha,\beta=1}^{m^{2}} k\left(\vec{\xi}_{\alpha}^{\sigma},\vec{\xi}_{\beta}^{\tau}\right) \int_{\operatorname{supp} \phi_{i}} \phi_{i}\left(\vec{x}\right)L_{\alpha}^{\sigma}\left(\vec{x}\right) \; d \vec{x} \int_{\operatorname{supp} \phi_{j}} \phi_{j}\left(\vec{y}\right) L_{\beta}^{\tau}\left(\vec{y}\right) \; d\vec{y}
\end{align*}
In fact, the expressions \(\int_{\operatorname{supp} \phi_{i}} \phi_{i}\left(\vec{x}\right)L_{\alpha}^{\sigma}\left(\vec{x}\right) \; d \vec{x}\) can be computed recursively starting from the finest level of the cluster tree, since for \(\tau\in\Sigma\left(\sigma\right)\) and \(\vec{x}\in Q_{\tau}\)
\begin{align*}
  L_{\alpha}^{\sigma}\left(\vec{x}\right)&=\sum_{\beta}L_{\alpha}^{\sigma}\left(\vec{\xi}_{\beta}^{\tau}\right) L_{\beta}^{\tau}\left(\vec{x}\right).
\end{align*}
This means that for all leaves \(\sigma=\left\{i\right\}\), and all \(1\leq\alpha\leq m^{d}\), the \emph{basis far-field coefficients}
\begin{align*}
  \int_{\operatorname{supp} \phi_{i}} \phi_{i}\left(\vec{x}\right)L_{\alpha}^{\sigma}\left(\vec{x}\right) \; d \vec{x}
\end{align*}
need to be evaluated (e.g. by \(m+1\)-th order Gaussian quadrature).
Moreover, the \emph{shift coefficients}
\begin{align*}
  L_{\alpha}^{\sigma}\left(\vec{\xi}_{\beta}^{\tau}\right)
\end{align*}
for \(\tau\in\Sigma\left(\sigma\right)\) must be evaluated, as well as the kernel approximations
\begin{align*}
  k\left(\vec{\xi}_{\alpha}^{\sigma},\vec{\xi}_{\beta}^{\tau}\right)
\end{align*}
for every admissible pair of clusters \(\left(\sigma,\tau\right)\).
We refer the reader to \cite{SauterSchwab2010_BoundaryElementMethods} for further details.

The consistency error of this approximation is given by the following theorem:
\begin{theorem}[\cite{SauterSchwab2010_BoundaryElementMethods}, Theorems 7.3.12 and 7.3.18]
  There exists \(\gamma\in(0,1)\) such that
  \begin{align*}
    \absAinsworthGlusa{k\left(\vec{x},\vec{y}\right))-k_{m}\left(\vec{x},\vec{y}\right)} &\leq \frac{C\gamma^{m}}{\operatorname{dist}\left(\sigma,\tau\right)^{d+2s}}.
  \end{align*}
  The consistency error between the bilinear form \(a(\cdot,\cdot)\) and the bilinear form \(a_{C}(\cdot,\cdot)\) of the panel clustering method is
  \begin{align*}
    \absAinsworthGlusa{a(u,v)-a_{C}(u,v)}
    &\leq C\gamma^{m}\left(1+2\eta\right)^{d+2s} C_{d,s}(h)\normAinsworthGlusa{u}_{L^{2}(\Omega)} \normAinsworthGlusa{v}_{L^{2}(\Omega)},
  \end{align*}
  where
  \begin{align*}
    C_{d,s}(h)
    &=
      \begin{cases}
        h^{-2} & \text{if } d=1 \text{ and }s< 1/2, \\
        h^{-2}\left(1+\absAinsworthGlusa{\log h}\right) & \text{if } d=1 \text{ and }s = 1/2, \\
        h^{-d-2s} & \text{otherwise}.
      \end{cases}
  \end{align*}
\end{theorem}

Again, by invoking Strang's Lemma, \(\mathcal{O}\left(h^{\ell-s}\right)\) convergence is retained if the interpolation order \(m\) satisfies
\begin{align*}
  m&\geq \frac{\left(\ell-s+2\right)\absAinsworthGlusa{\log h}}{\absAinsworthGlusa{\log \gamma}} && \text{if } d=1 \text{ and } s<1/2\\
  m&\geq \frac{\left(\ell-s+2\right)\absAinsworthGlusa{\log h}+\log\left(1+\absAinsworthGlusa{\log h}\right)}{\absAinsworthGlusa{\log \gamma}} && \text{if } d=1 \text{ and } s=1/2\\
  m&\geq \frac{\left(\ell-s+d+2s\right)\absAinsworthGlusa{\log h}}{\absAinsworthGlusa{\log \gamma}} && \text{otherwise}.
\end{align*}

By following the arguments in \cite{SauterSchwab2010_BoundaryElementMethods}, it can be shown that the number of near field entries, i.e. the entries that need to be assembled using the quadrature rules described in \Cref{sec:comp-entr-stiffn}, scales linearly in \(n\).
The same conclusion holds for the number of far field cluster pairs.
Since the four dimensional integral contributions \(a^{K\times \tilde{K}}\) are evaluated using Gaussian quadrature rules with at most \(k\sim\log n\) quadrature nodes per dimension, the assembly of the near field contributions scales with \(n\log^{2d}n\) .
The far field kernel approximations and the shift coefficients have size \(m^{2d}\sim \log^{2d}n\), and are also calculated in \(\log^{2d}n\) complexity.
This means that all the kernel approximations and shift coefficients are obtained in \(n\log^{2d}n\) time.
Finally, the \(nm^{d}\) basis far-field coefficients require the evaluation of integrals using \(m+1\)-th order Gaussian quadrature, leading to a complexity of \(n\log^{2d}n\) as well.
The overall complexity of the panel clustering method is therefore \(\mathcal{O}\left(n\log^{2d}n\right)\), and the sparse approximation requires \(\mathcal{O}\left(n\log^{2d}n\right)\) memory.
In practice, this means that the assembly of the near-field matrix dominates the other steps but involves only local computations.

The computation of the matrix-vector product involving upward and downward recursion in the cluster tree and multiplication by the kernel approximations can also be shown to scale with \(\mathcal{O}\left(n\log^{2d}n\right)\).

As an aside, we note that one could also opt to use a conventional dense approximation of the discretised fractional Laplacian such as the ``hybrid'' scheme described in \cite{GrahamHackbuschEtAl2000_HybridGalerkinBoundaryElements} which reduces the far field computation to the computation of a ``Nystr\"om-type'' approximation.
While the complexity of this approach still scales as \(\mathcal{O}\left(n^{2}\right)\), the constant is significantly smaller than if the dense matrix were to be used.

We illustrate the above results by assembling both the full matrix as well as its sparse approximation on the unit disk for fractional orders \(s=0.25\) and \(s=0.75\).
The memory usage of the matrices are compared in \Cref{fig:memory}.
For low number of degrees of freedom, none of the cluster pairs are admissible, so the full matrix and its approximation have the same size.
Starting with roughly 2000 degrees of freedom, the memory footprint of the sparse approximate starts to follow the \(n\log^{4}n\) curve and therefore outperforms the full assembly.
\begin{figure}
  \centering
  % created using plot_fracLaplOrder.py
  \includegraphics{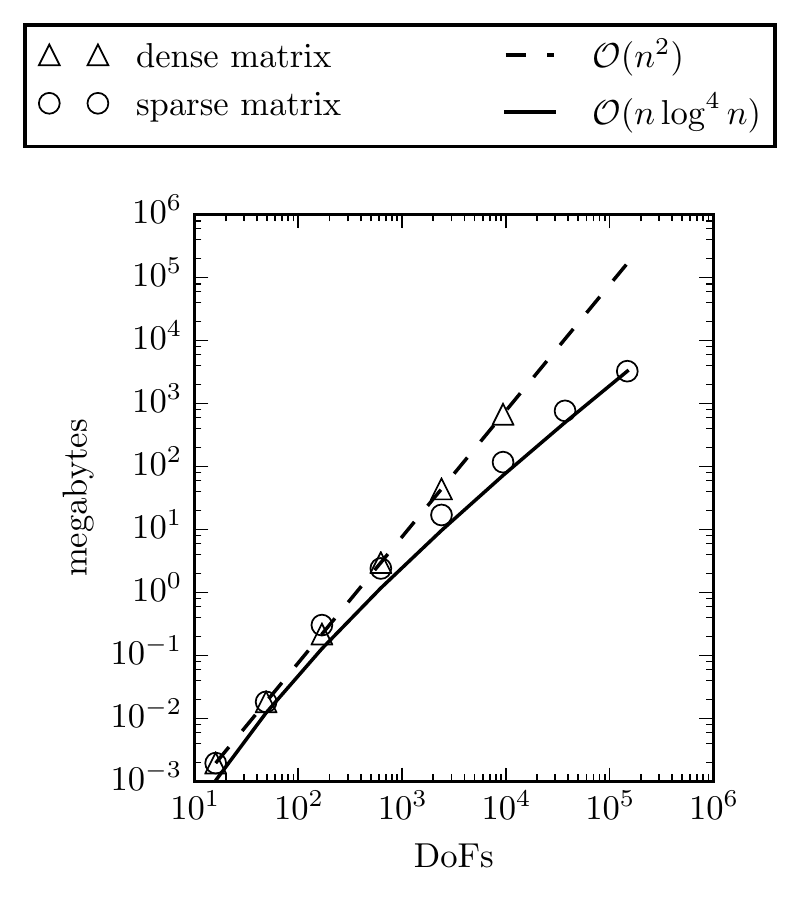}
  \includegraphics{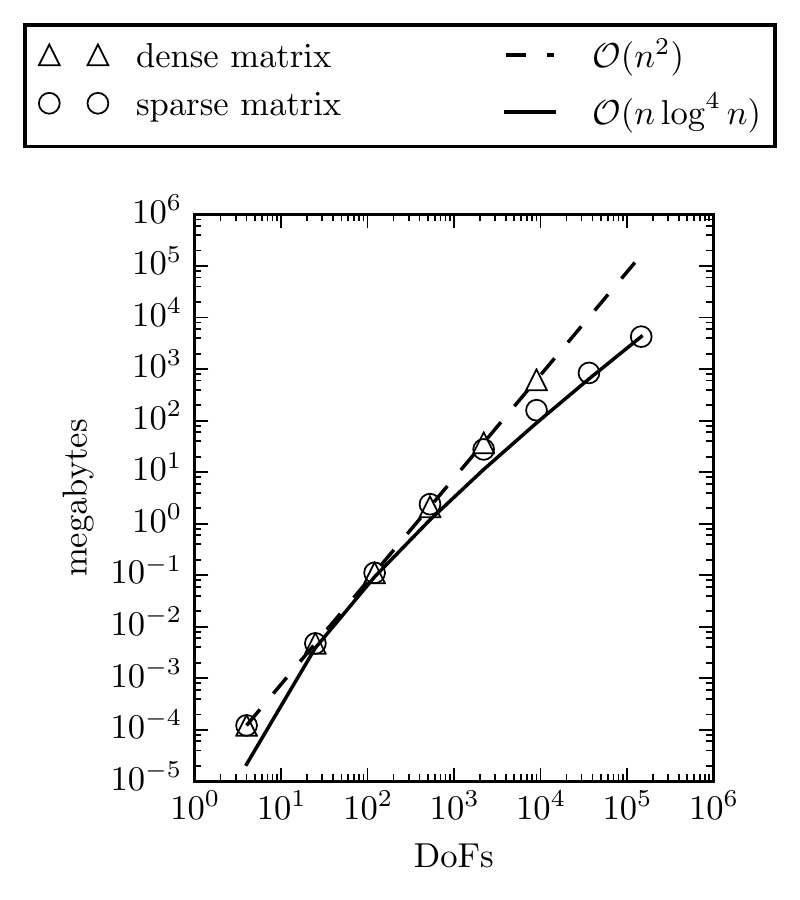}
  \caption{
    Memory usage of the dense matrix and its sparse approximation.
    \(s=0.25\) (\emph{top}), \(s=0.75\) (\emph{bottom}).
    While the dense matrix uses \(n^{2}\) floating-point numbers, the sparse approximation can be seen to require only \(\mathcal{O}\left(n\log^{4}n\right)\) memory.
    At roughly 2000 unknowns, the memory footprint of the sparse approximation separates from the \(\mathcal{O}\left(n^{2}\right)\) curve.
  }
  \label{fig:memory}
\end{figure}
The same behaviour can be observed for the assembly times, as seen in \Cref{fig:timeAssembly}.
\begin{figure}
  \centering
  % created using plot_fracLaplOrder.py
  \includegraphics{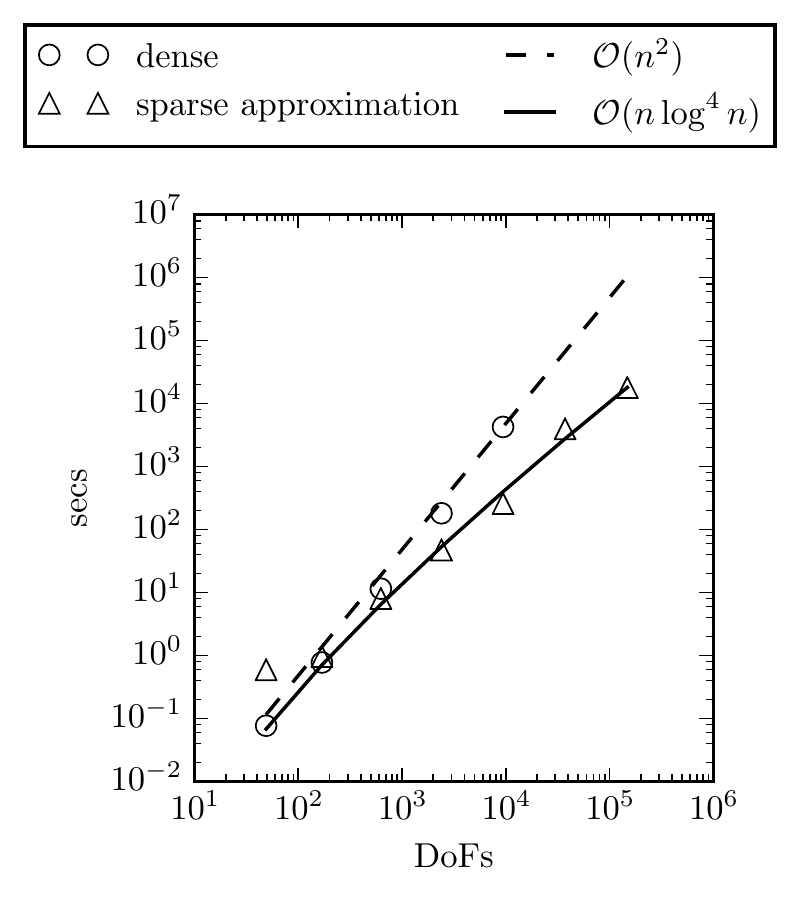}
  \includegraphics{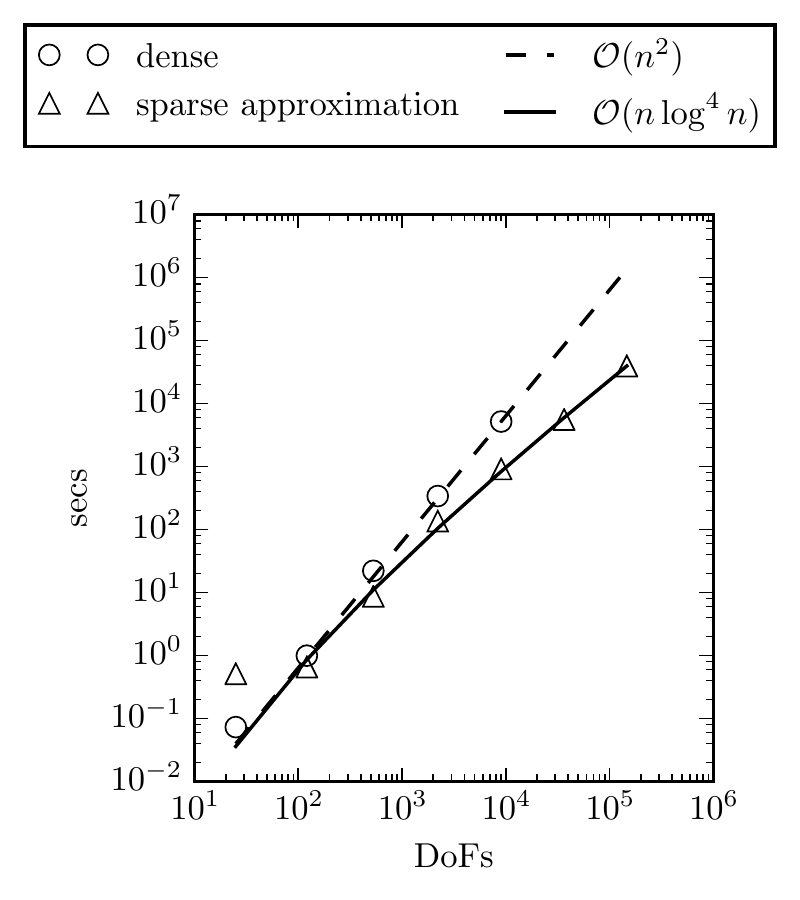}
  \caption{
    Assembly time of the dense matrix and its sparse approximation.
    \(s=0.25\) (\emph{top}), \(s=0.75\) (\emph{bottom}).
    The time to assemble the full matrix grows quadratically in the number of unknowns, whereas the sparse approximation starts to follow the \(n\log^{4}n\) curve at about 2000 degrees of freedom.
  }
  \label{fig:timeAssembly}
\end{figure}

\section{Applications}
\label{sec:numerical-experiments}

\subsection{Fractional Poisson Equation}
\label{sec:fract-poiss-equat}

We consider the fractional Poisson problem
\begin{align*}
  \begin{aligned}
    \left(-\Delta\right)^{s}u &= f &\text{in }\Omega, \\
    u&= 0 & \text{in }\Omega^{c}
  \end{aligned}
\end{align*}
on the unit disk \(\Omega = \left\{\vec{x}\in\mathbb{R}^{2}\mid\absAinsworthGlusa{\vec{x}}\leq 1\right\}\).
The discretised fractional Poisson problem then reads
\begin{align}
  \boldsymbol{A}^{s}\vec{u} &= \vec{b},\label{eq:5}
\end{align}
where \(u_{h}=\sum_{i=1}^{n}u_{i}\phi_{i}\in V_{h}\) is the approximation to the solution \(u\), and \(b_{i}=\left\langle f,\phi_{i}\right\rangle\).

Triangulations of the disc are obtained through uniform refinement of a uniform initial mesh.
After each refinement, the boundary nodes are projected onto the unit circle, resulting in triangulations of the type shown in \Cref{fig:disk}.

\begin{figure}
  \centering
  % created using dodo.py
  \includegraphics{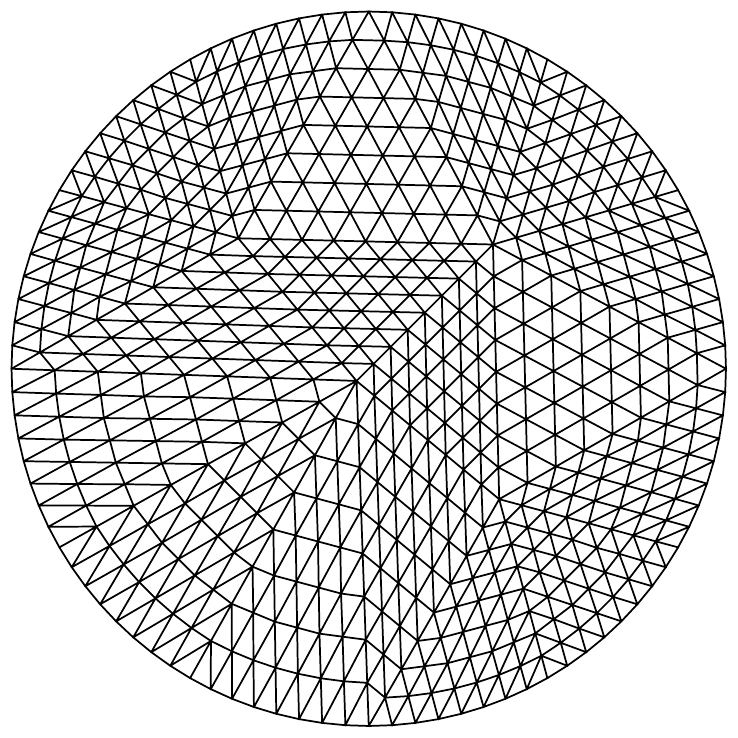}
  \caption{A quasi-uniform triangulation of the disc domain, obtained through uniform refinement followed by projection of the resulting boundary nodes back onto the unit circle.}
  \label{fig:disk}
\end{figure}

We first consider the test case introduced in \Cref{sec:finite-elem-appr} where \(f=1\) with analytic solution \cite{Getoor1961_FirstPassageTimesSymmetric} given by
\begin{align*}
  u^{s}\left(\vec{x}\right) := \frac{2^{-2s}}{\Gamma\left(1+s\right)^{2}} \left(1-\absAinsworthGlusa{\vec{x}}^{2}\right)^{s}.
\end{align*}
Both the full matrix and its sparse approximation are assembled for \(s\in\left\{0.25, 0.75\right\}\), and \cref{eq:5} is solved using LAPACK's \texttt{dgesv} routine and a multigrid solver in the dense case, and multigrid and conjugate gradient methods in the sparse case.
Two steps of pre- and postsmoothing by Jacobi iteration are used on every level of the multigrid solver.
Recall that solutions for \(s=0.25\) and \(s=0.75\) were shown in \Cref{fig:solutions}.
In \Cref{fig:errorHs,fig:errorL2}, the discretisation error is plotted in \(\widetilde{H}^{s}\left(\Omega\right)\) and in \(L^{2}\)-norm.
It can be seen that the rates predicted by \Cref{thm:Hsconv,thm:L2conv} of \(h^{1/2}\) and \(h^{1/2+\min(1/2,s)}\) are indeed obtained, and that the error curves for the full matrix and its sparse approximation are essentially indistinguishable.

\begin{figure}
  \centering
  % created using plot_fracLaplOrder.py
  \includegraphics{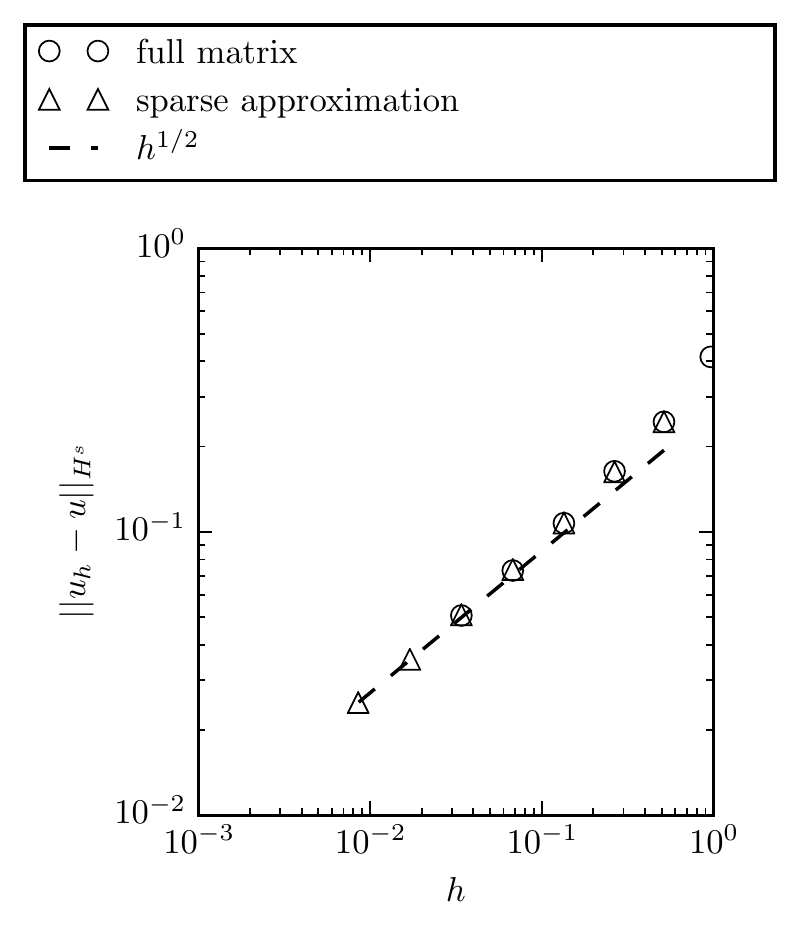}
  \includegraphics{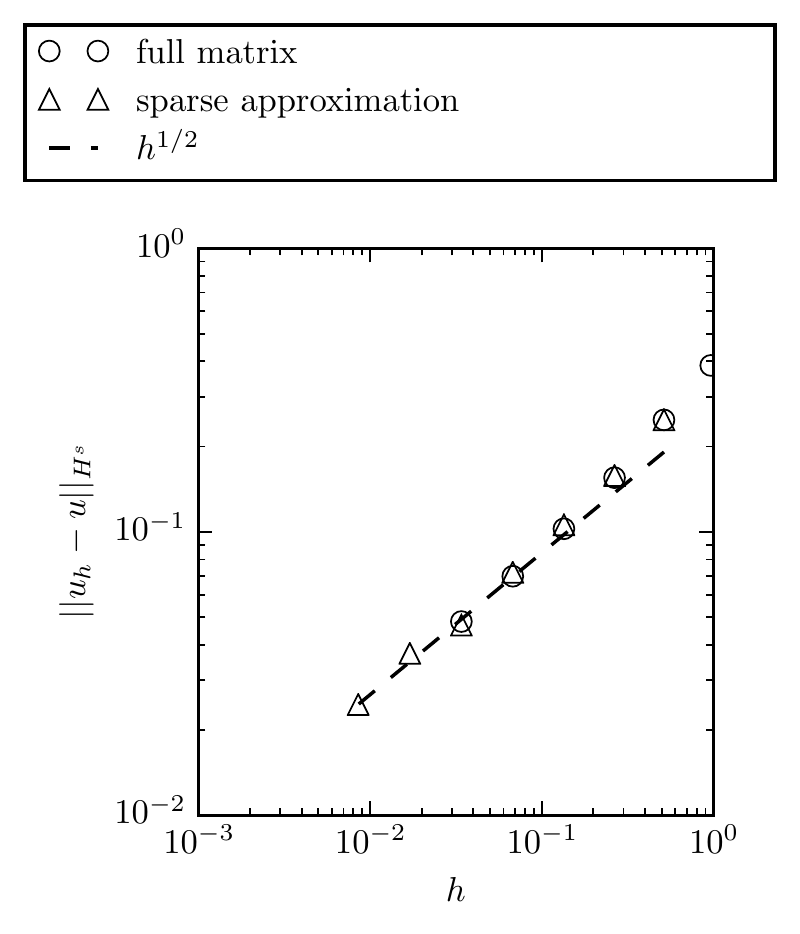}
  \caption{
    Error \(\normAinsworthGlusa{u^{s}-u_{h}}_{\widetilde{H}^{s}\left(\Omega\right)}\) for \(s=0.25\) (\emph{top}) and \(s=0.75\) (\emph{bottom}) in the case of solutions with singular behaviour close to the boundary.
    Both the full matrix and its sparse approximation are shown to achieve the predicted rate of \(h^{1/2}\).
  }
  \label{fig:errorHs}
\end{figure}
\begin{figure}
  \centering
  % created using plot_fracLaplOrder.py
  \includegraphics{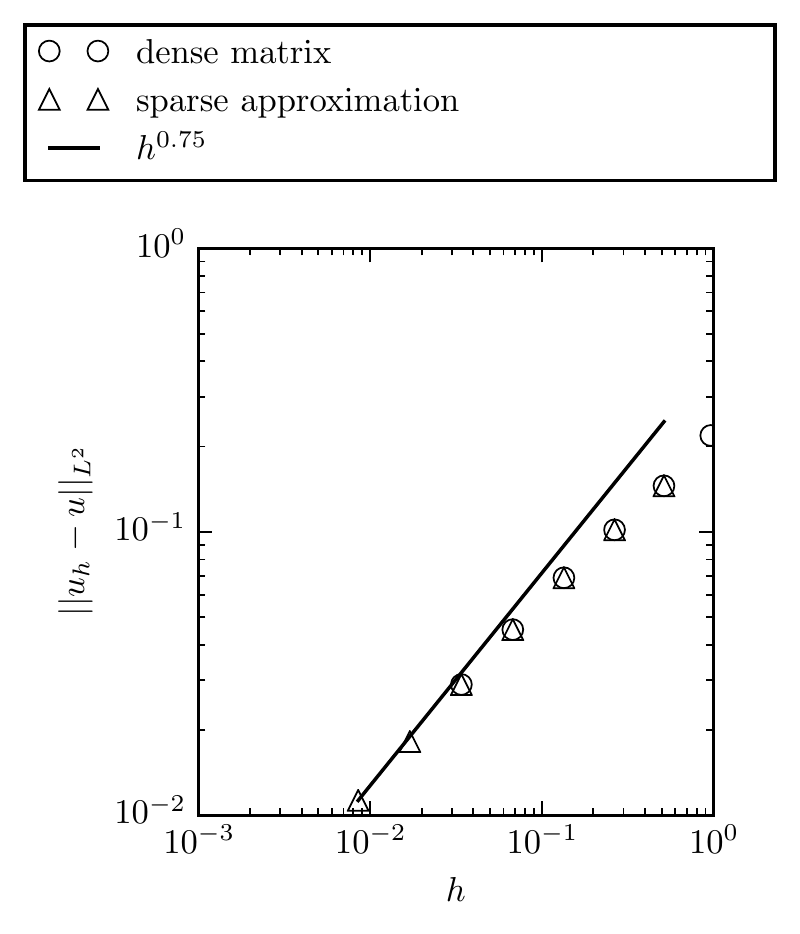}
  \includegraphics{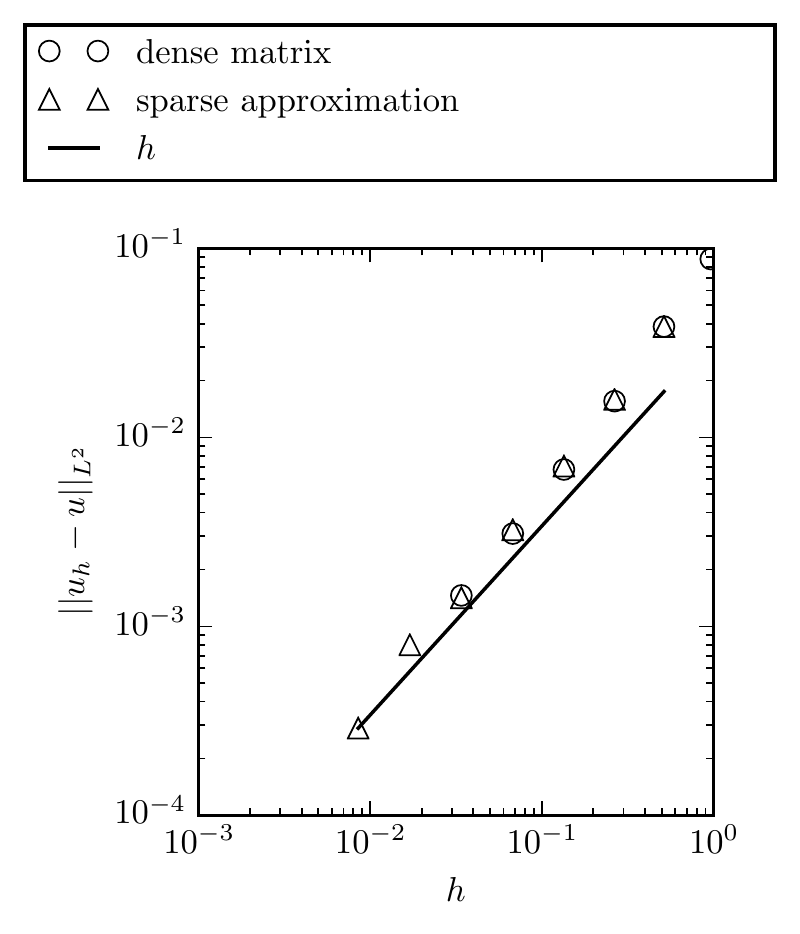}
  \caption{
    Error \(\normAinsworthGlusa{u^{s}-u_{h}}_{L^{2}}\) for \(s=0.25\) (\emph{top}) and \(s=0.75\) (\emph{bottom}) in the case of solutions with singular behaviour close to the boundary.
    Both the full matrix and its sparse approximation are shown to achieve the predicted rate of \(h^{1/2+\min\left\{s,1/2\right\}}\).
  }
  \label{fig:errorL2}
\end{figure}

For a second example, the right-hand side \(f\) is chosen such that \(u=1-\absAinsworthGlusa{\vec{x}}^{2}\in H^{2}\left(\Omega\right)\).
The action of \(f\) on \(v\in V_{h}\) is approximated by
\begin{align*}
  \left(f,v\right)&= a(I_{\underline{h}}u, v),
\end{align*}
where \(I_{\underline{h}}\) is the interpolation operator onto a highly refined mesh with \(\underline{h}<h\).
The resulting consistency error in this case is
\begin{align*}
  \sup_{v}\frac{\absAinsworthGlusa{a(u,v)-a(I_{\underline{h}}u,v)}}{\normAinsworthGlusa{v}_{\widetilde{H}^{s}\left(\Omega\right)}}
  \leq C\normAinsworthGlusa{u-I_{\underline{h}}u}_{\widetilde{H}^{s}\left(\Omega\right)}
  \leq C \underline{h}^{2-s}\absAinsworthGlusa{u}_{H^{2}}.
\end{align*}
Therefore, if \(\underline{h}\) is sufficiently smaller than \(h\), the consistency error will be negligible compared to the discretisation error.

The dependency of the error on the mesh size \(h\) can be seen in \Cref{fig:smooth_error}.
The discretisation error decays as \(h^{2-s}\) in \(\widetilde{H}^{s}\left(\Omega\right)\)-norm, and as \(h^{2}\) in \(L^{2}\)-norm, which are the optimal orders that we would expect based on estimate \eqref{eq:16}.

\begin{figure}
  \centering
  % created using dodo.py
  \includegraphics{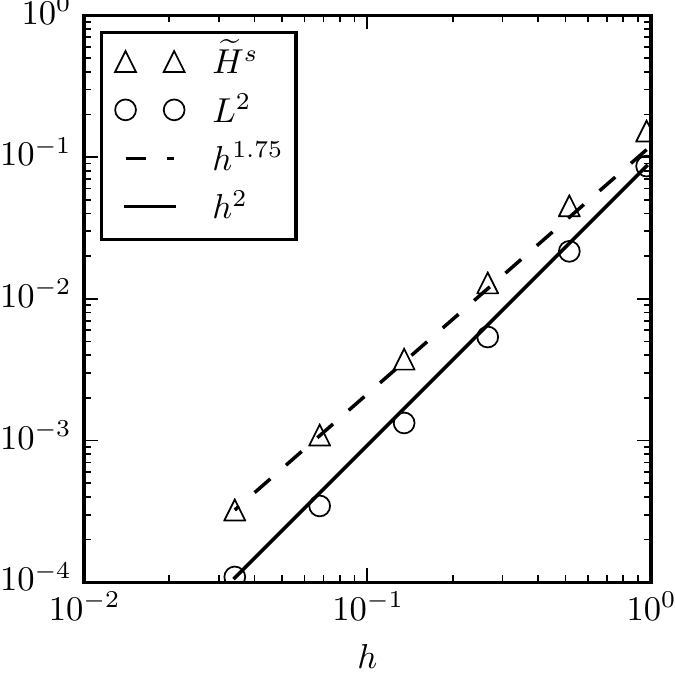}
  \includegraphics{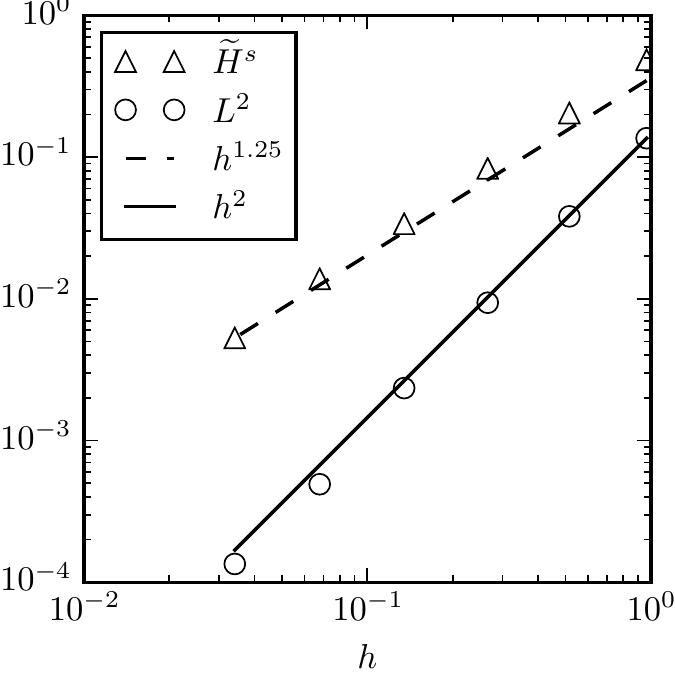}
  \caption{
    Errors \(\normAinsworthGlusa{u-u_{h}}_{\widetilde{H}^{s}\left(\Omega\right)}\) and \(\normAinsworthGlusa{u-u_{h}}_{L^{2}}\) for \(s=0.25\) (\emph{top}) and \(s=0.75\) (\emph{bottom}) in the case of a smooth solution \(u(\vec{x})=1-\absAinsworthGlusa{\vec{x}}^{2}\in H^{2}\left(\Omega\right)\).
    Optimal orders are achieved both in \(\widetilde{H}^{s}\left(\Omega\right)\)- and \(L^{2}\)-norm.
  }
  \label{fig:smooth_error}
\end{figure}

Summarising the results of \Cref{sec:solut-syst-involv,sec:cluster-method}, we expect different solvers for the fractional Laplacian to have complexities as given in \Cref{tab:complexities}.
\begin{table*}[h]
  \centering
  \normalsize
  \begin{tabular}{lcc}
    Method & dense matrix & sparse approximation\\
    \hline
    Dense Solver& \(n^{3}\)& -- \\
    Conjugate Gradient & \(n^{2+s/d}\)& \(n^{1+s/d}\left(\log n\right)^{2d}\)\\
    Multigrid& \(n^{2}\) & \(n\left(\log n\right)^{2d}\)
  \end{tabular}
  \caption{Asymptotic complexities of different solvers for the discretised fractional Poisson problem \(\boldsymbol{A}^{s}\vec{u}=\vec{b}\).}\label{tab:complexities}
\end{table*}
The timings for the different combinations of dense or sparse matrix with a solver are shown in \Cref{fig:timeSolve}.
It can be observed that the sparse approximation asymptotically outperforms the dense solvers.
Moreover, for the larger value of \(s\), the multigrid solver starts to outperform the conjugate gradient method for increasingly smaller numbers of unknowns as one would expect based on earlier arguments.
\begin{figure}
  \centering
  % created using plot_fracLaplOrder.py
  \includegraphics{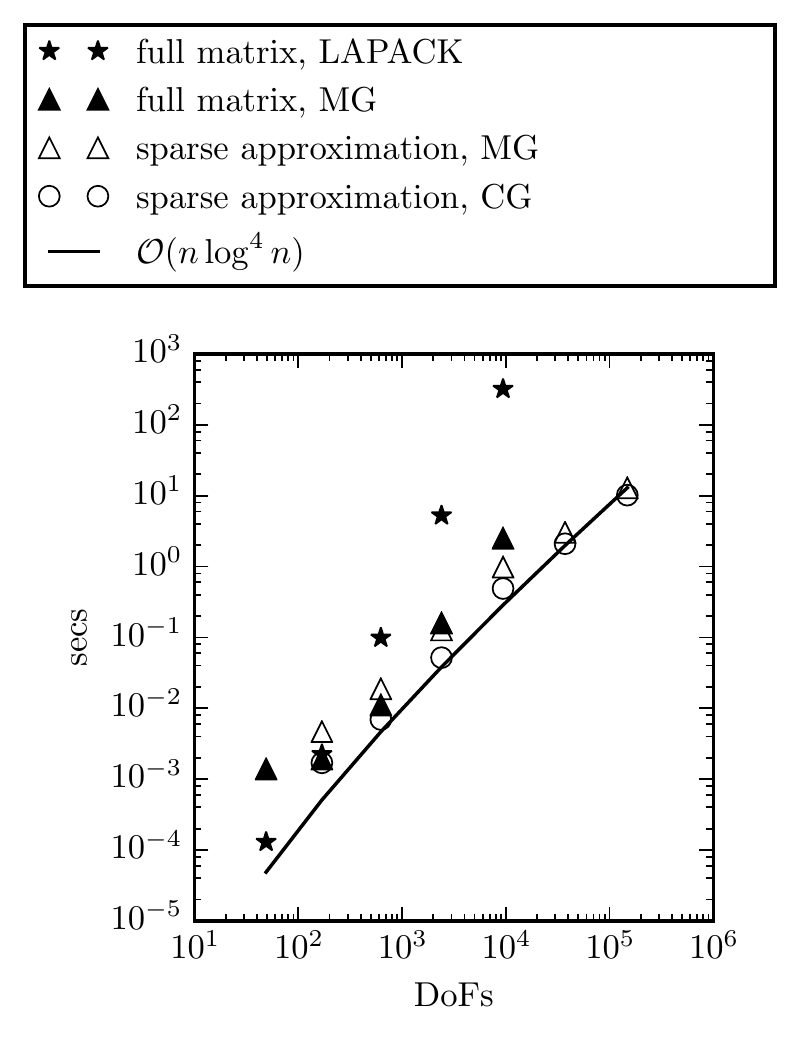}
  \includegraphics{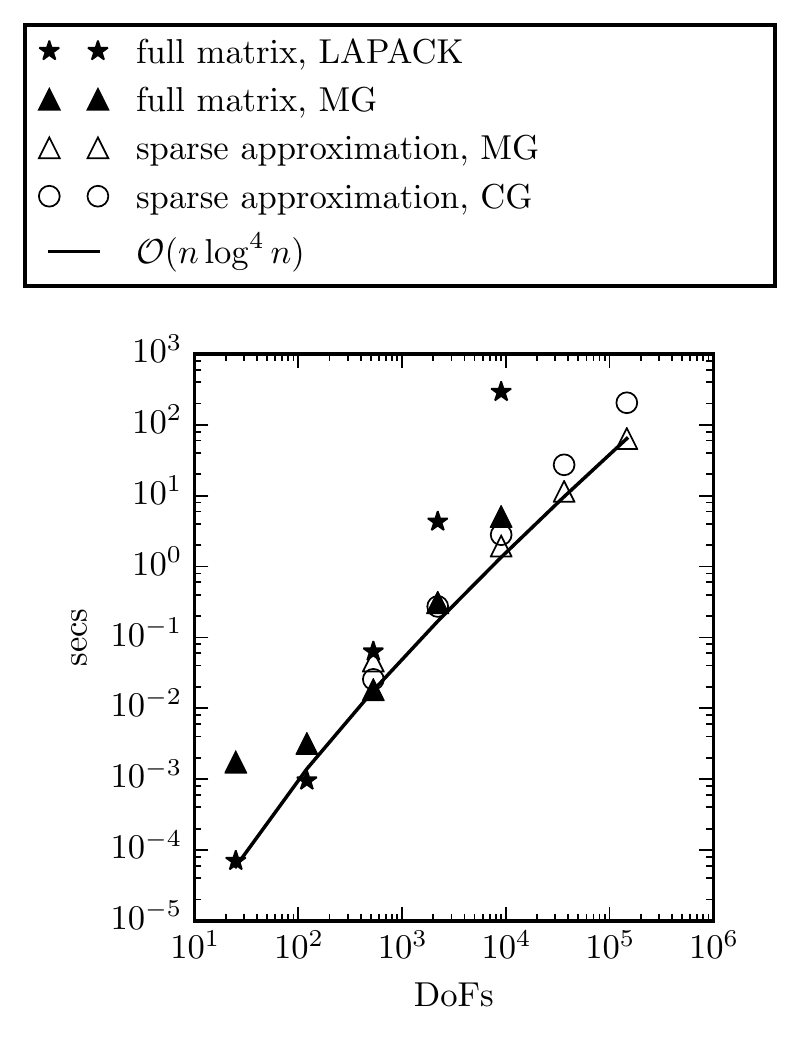}
  \caption{
    Solution time for the fractional Laplacian using different solvers and the full matrix and its sparse approximation for \(s=0.25\) (\emph{top}) and \(s=0.75\) (\emph{bottom}).
    The solvers using the full matrix are outperformed by the ones based on the sparse approximation.
    For larger fractional order \(s\), the break-even between conjugate gradient and multigrid iteration occurs at a lower number of unknowns.
  }
  \label{fig:timeSolve}
\end{figure}

\subsection{Fractional Heat Equation}
\label{sec:fract-heat-equat}

The fractional heat equation is given by
\begin{align*}
  \begin{aligned}
    u_{t}+\left(-\Delta\right)^{s}u & = f && \text{in }\Omega, \\
    u &=0 && \text{in } \Omega^{c}.
  \end{aligned}
\end{align*}
We propose to approximate the problem using an implicit method in time.
The simplest such scheme is the backward Euler method
\begin{align*}
  \left(\boldsymbol{M}+\Delta t~\boldsymbol{A}^{s}\right)\vec{u}^{k+1} &= \boldsymbol{M}\vec{u}^{k} + \Delta t \vec{f}^{k+1},
\end{align*}
where \(u(\cdot, k\Delta t)\approx \sum_{i}u_{i}^{k}\phi_i\) and \(f^{k}_{i}=\left(f(\cdot,k\Delta t), \phi_{i}\right)\).

More generally, let us assume that a scheme of order \(\alpha\) is used in time.
In order to obtain optimal convergence in \(L^{2}\)-norm, in view of \Cref{thm:L2conv}, we shall choose \(\Delta t^{\alpha}\sim h^{1/2+\min(1/2,s)}\), i.e.
\begin{align*}
  \Delta t_{L^{2}} \sim h^{\min\left(2,1+2s\right)/(2\alpha)}.
\end{align*}
On the other hand, if optimal \(\widetilde{H}^{s}\left(\Omega\right)\)-convergence is desired, we need \(\Delta t_{\widetilde{H}^{s}\left(\Omega\right)}\sim h^{1/(2\alpha)}\), see \Cref{thm:Hsconv}.
Consequently, if an order \(\alpha\) scheme is used for time stepping, with optimal time step \(\Delta t_{L^{2}}\) or \(\Delta t_{\widetilde{H}^{s}\left(\Omega\right)}\), we find by \Cref{lem:condFracLaplTime} that the condition numbers of the iteration matrix satisfy
\begin{align*}
  \kappa\left(\boldsymbol{M}+\Delta t_{L^{2}}~\boldsymbol{A}^{s}\right)
  &\leq C \left(1+h^{\min\left(2,1+2s\right)/(2\alpha)-2s}\right),\\
  \kappa\left(\boldsymbol{M}+\Delta t_{\widetilde{H}^{s}\left(\Omega\right)}~\boldsymbol{A}^{s}\right)
  & \leq C \left(1+h^{1/(2\alpha)-2s}\right).
\end{align*}
In particular, in the \(L^{2}\) case, this shows that the condition number will not grow at all as the mesh size decreases if \(s\in (0,1/\left(4\alpha-2\right)]\).
For fractional orders \(s\) that are slightly larger than  \(1/(4\alpha-2)\), the condition number only grows very slowly as the mesh size is decreased.
The larger the fractional order, the faster the linear system becomes ill-conditioned.
In the  \(\widetilde{H}^{s}\left(\Omega\right)\) case, the condition number of the linear system grows as the mesh size is decreased for \(s> 1/\left(4\alpha\right)\).

We illustrate the consequences of the above result in the case of a second order accurate time stepping scheme (\(\alpha=2\)), and for \(s=0.25\) and \(s=0.75\).
In the case of \(s=0.25\), \(\Delta t_{L^{2}}\sim h^{3/8}\) and \(\kappa\left(\boldsymbol{M}+\Delta t_{L^{2}}~\boldsymbol{A}^{s}\right)\sim 1+h^{-1/8}\).
This suggests that the conjugate gradient method will deliver good results for a wide range of mesh sizes \(h\), as the number of iterations will only grow as \(\sqrt{\kappa\left(\boldsymbol{M}+\Delta t_{L^{2}}~\boldsymbol{A}^{s}\right)}\sim h^{-1/16}\).
The convergence of the multigrid method does not depend on the condition number and is essentially independent of \(h\).
This is indeed what is observed in the top part of \Cref{fig:CG-MG-time}.
In \Cref{fig:CG-MG-iter}, the number of iterations is shown.
It can be observed that for \(s=0.25\) both the multigrid and the conjugate gradient solver require an essentially constant number of iterations for varying values of \(\Delta t\).

On the other hand, for \(s=0.75\), \(\Delta t_{L^{2}}\sim h^{1/2}\) and \(\kappa\left(\boldsymbol{M}+\Delta t_{L^{2}}~\boldsymbol{A}^{s}\right)\sim 1+h^{-1}\).
Therefore, the condition number increases a lot faster as \(h\) goes to zero, and we expect that multigrid asymptotically outperforms the CG solver.
This is indeed what is observed in \Cref{fig:CG-MG-iter,fig:CG-MG-time}.

\begin{table*}
  \centering
  \normalsize
  \begin{tabular}{l|c|c}
    Method & \(\Delta t = \Delta t_{L^{2}}\) & \(\Delta t = \Delta t_{\widetilde{H}^{s}\left(\Omega\right)}\) \\
    \hline
    Conjugate Gradient & \(n^{1+2s/d-\min\left(2,1+2s\right)/(2\alpha d)} \left(\log n\right)^{2d}\)& \(n^{1+2s/d-1/(2\alpha d)} \left(\log n\right)^{2d}\)\\
    Multigrid & \(n \left(\log n\right)^{2d}\)& \(n \left(\log n\right)^{2d}\)
  \end{tabular}
  \caption{Complexity of different solvers for \(\left(\boldsymbol{M}+\Delta t\boldsymbol{A}^{s}\right)\vec{u}=\vec{b}\) for \(\Delta t=\Delta t_{L^{2}}\) and \(\Delta t=\Delta t_{\widetilde{H}^{s}\left(\Omega\right)}\) for an \(\alpha\)-order time stepping scheme.}
  \label{tab:solverTimestepping}
\end{table*}
The complexities of the different solvers for different choices of time step size are summarised in \Cref{tab:solverTimestepping}.

\begin{figure}
  \centering
  % created using plot_fracLaplTime.py
  \includegraphics{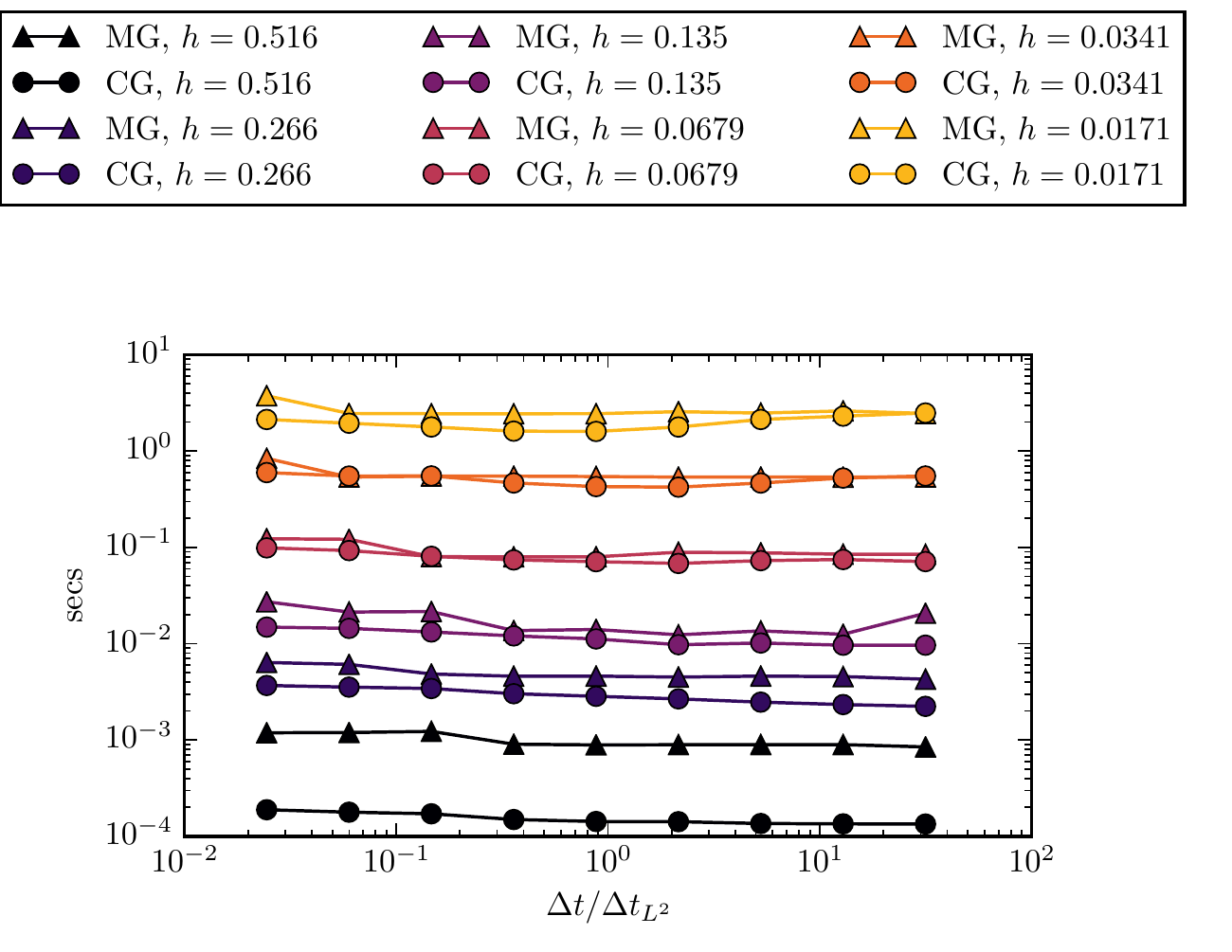}
  \includegraphics{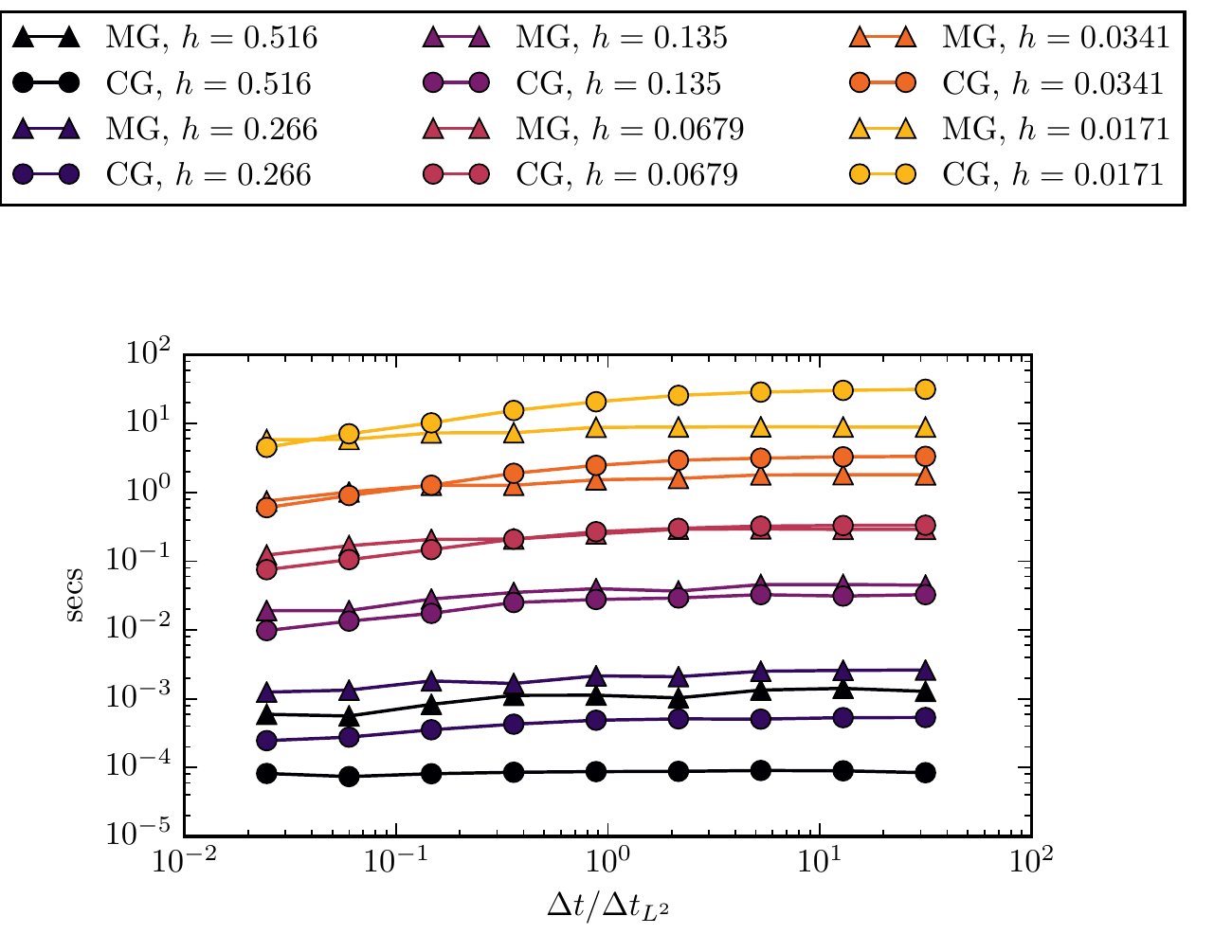}
  \caption{
    Timings in seconds for CG and MG depending on \(\Delta t\) for \(s=0.25\) (\emph{top}) and \(s=0.75\) (\emph{bottom}).
    It can be observed that, for \(s=0.25\), the conjugate gradient method is essentially on par with the multigrid solver.
    For \(s=0.75\), the multigrid solver asymptotically outperforms the conjugate gradient method, since the condition number \(\kappa\left(\boldsymbol{M}+\Delta t_{L^{2}}\boldsymbol{A}^{s}\right)\) grows as \(h^{-1}\).
  }
  \label{fig:CG-MG-time}
\end{figure}

\begin{figure}
  \centering
  % created using plot_fracLaplTime.py
  \includegraphics{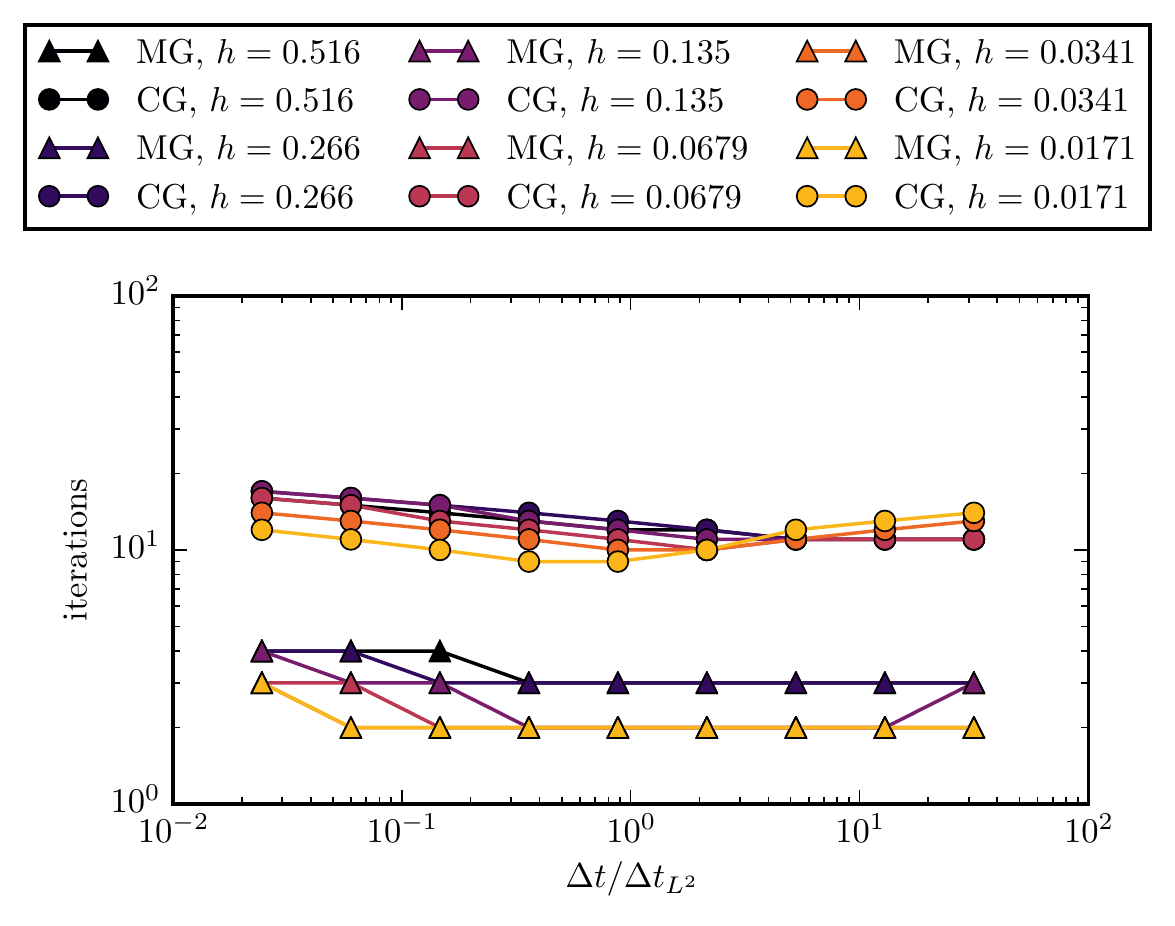}
  \includegraphics{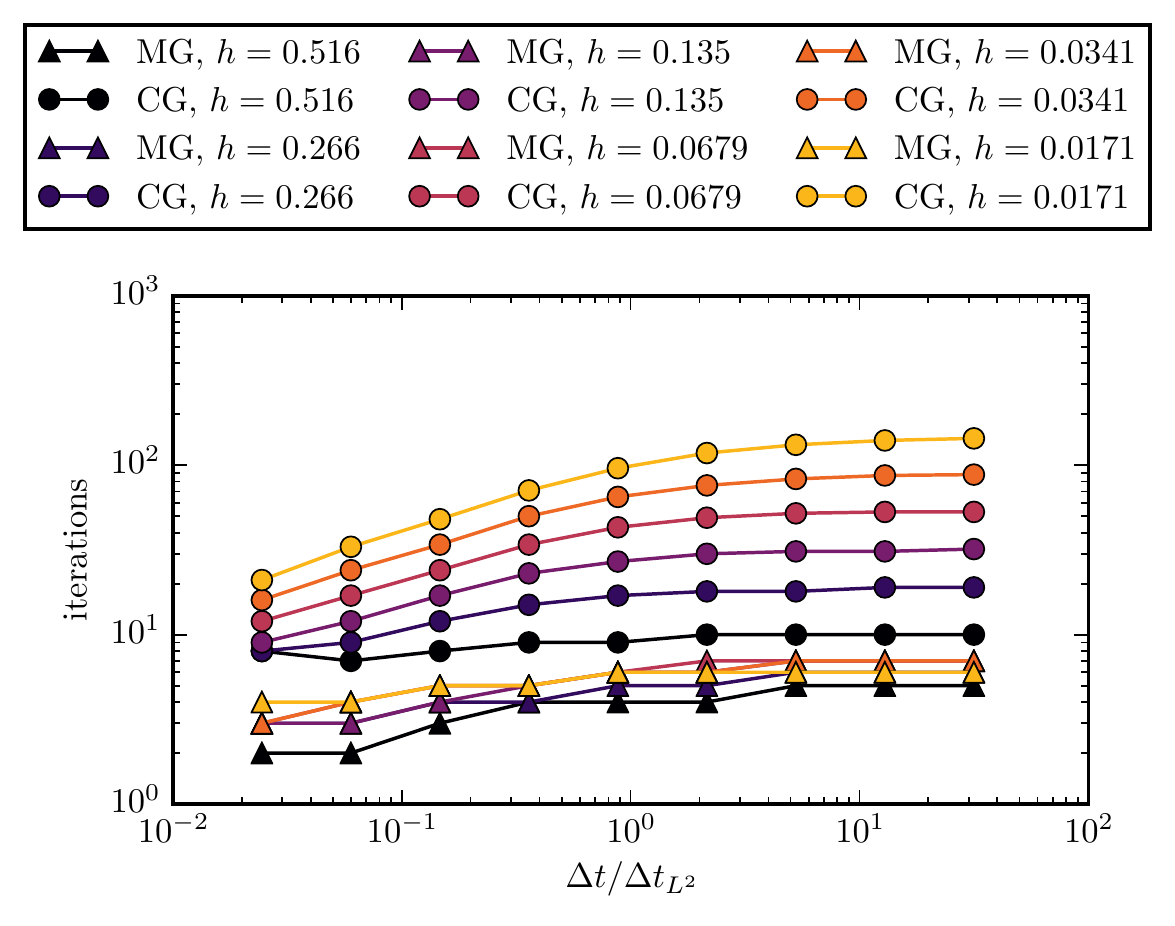}
  \caption{
    Number of iterations for CG and MG depending on \(\Delta t\) for \(s=0.25\) (\emph{top}) and \(s=0.75\) (\emph{bottom}).
    For \(s=0.25\), the number of iterations is essentially independent of \(\Delta t\).
    For \(s=0.75\), the number of iterations of the multigrid solver is independent of \(\Delta t\), but the iterations count for conjugate gradient grows with \(h^{-1/2}\).
  }
  \label{fig:CG-MG-iter}
\end{figure}

\subsection{Fractional Reaction-Diffusion Systems}
\label{sec:fract-react-diff}

In \cite{GolovinMatkowskyEtAl2008_TuringPatternFormationBrusselator}, a space-fractional Brusselator model was analysed and compared to the classical integer-order case.
The coupled system of equations is given by
\begin{align*}
  \frac{\partial X}{\partial t} &= -D_{X}\left(-\Delta\right)^{\alpha} X + A - (B+1)X + X^{2}Y, \\
  \frac{\partial Y}{\partial t} &= -D_{Y}\left(-\Delta\right)^{\beta} Y +BX - X^{2}Y.
\end{align*}
Here, \(D_{X}\) and \(D_{Y}\) are diffusion coefficients, \(A\) and \(B\) are reaction parameters, and \(\alpha\) and \(\beta\) determine the type of diffusion.
By rewriting the solutions as deviations from the stationary solution \(X=A\), \(Y=B/A\) and rescaling, one obtains
\begin{align}
  \frac{\partial u}{\partial t} &= -\left(-\Delta\right)^{\alpha} u + (B-1)u + Q^{2}v +\frac{B}{Q}u^{2} + 2Quv + u^{2}v, \label{eq:10}\\
  \eta^{2}\frac{\partial v}{\partial t} &= -\left(-\Delta\right)^{\beta} v - B u - Q^{2}v -\frac{B}{Q}u^{2} - 2Quv - u^{2}v, \label{eq:11}
\end{align}
with \(\eta=\sqrt{D_{Y}/D_{X}^{\beta/\alpha}}\) and \(Q=A\eta\).

In \cite{GolovinMatkowskyEtAl2008_TuringPatternFormationBrusselator} the equations were augmented with periodic boundary conditions and approximated using a pseudospectral method for various different parameter combinations.
Here, thanks to the foregoing developments, we have the flexibility to handle more general domains and, in particular, we consider the case where \(\Omega\) corresponds to a Petri-dish, i.e. \(\Omega=\left\{\vec{x}\in\mathbb{R}^{2}\mid\absAinsworthGlusa{\vec{x}}\leq 1\right\}\) is the unit disk.
We solve the above set of equations using a second order accurate IMEX scheme proposed by Koto \cite{Koto2008_ImexRunge}, whose Butcher tableaux are given by \Cref{tab:imex}.
\begin{table}
  \centering
  \begin{tabular}{r|cccc}
    0& 0&&& \\
    1&0&1&& \\
    1/2&0&-1/2&1& \\
    1 &0&-1&1&1 \\
    \hline
     &0&-1&1&1
  \end{tabular}
  \begin{tabular}{r|cccc}
    0& &&& \\
    1&1&&& \\
    1/2&0&0&& \\
    1 &0&0&1& \\
    \hline
     &0&0&1&0
  \end{tabular}
  \caption{IMEX scheme by Koto. Implicit scheme on the left, explicit on the right.}\label{tab:imex}
\end{table}
The diffusive parts are treated implicitly and therefore require the solution of several systems all of which are of the type \(\boldsymbol{M}+c\Delta t \boldsymbol{A}^{s}\) with appropriate values of \(c\).

In order to verify the correct convergence behaviour, we add forcing functions \(f\) and \(g\) to the system,
chosen such that the analytic solution is given by
\begin{align*}
  u&= \eta\sin(t) u^{s}(\vec{x}), \\
  v&= \eta^{-1}\cos(2t) u^{s}(\vec{x}),
\end{align*}
for suitable initial conditions, where \(u^{s}\) is the solution of the fractional Poisson problem with constant right-hand side.
We take \(\alpha=\beta=0.75\), and choose \(\Delta t\sim h^{1/2}\), since we already saw that the rate of the spatial approximation in \(L^{2}\)-norm is of order \(h\).
We measure the error as
\begin{align*}
  e_{L^{2}}^{u} &= \max_{0\leq t_{i}\leq 10}\normAinsworthGlusa{u(t_{i},\cdot)-u_{h}^{i}}_{L^{2}},
  &e_{L^{2}}^{v} &= \max_{0\leq t_{i}\leq 10}\normAinsworthGlusa{v(t_{i},\cdot)-v_{h}^{i}}_{L^{2}}, \\
  e_{\widetilde{H}^{s}\left(\Omega\right)}^{u} &= \max_{0\leq t_{i}\leq 10}\normAinsworthGlusa{u(t_{i},\cdot)-u_{h}^{i}}_{\widetilde{H}^{s}\left(\Omega\right)},
  &e_{\widetilde{H}^{s}\left(\Omega\right)}^{v} &= \max_{0\leq t_{i}\leq 10}\normAinsworthGlusa{v(t_{i},\cdot)-v_{h}^{i}}_{\widetilde{H}^{s}\left(\Omega\right)}.
\end{align*}
From the error plots in \Cref{fig:errorBrusselator}, it can be observed that \(e_{L^{2}}\sim h\) and \(e_{V}\sim h^{1/2}\), as expected.
\begin{figure}
  \centering
  % created using plot_conv_brusselator.py
  \includegraphics{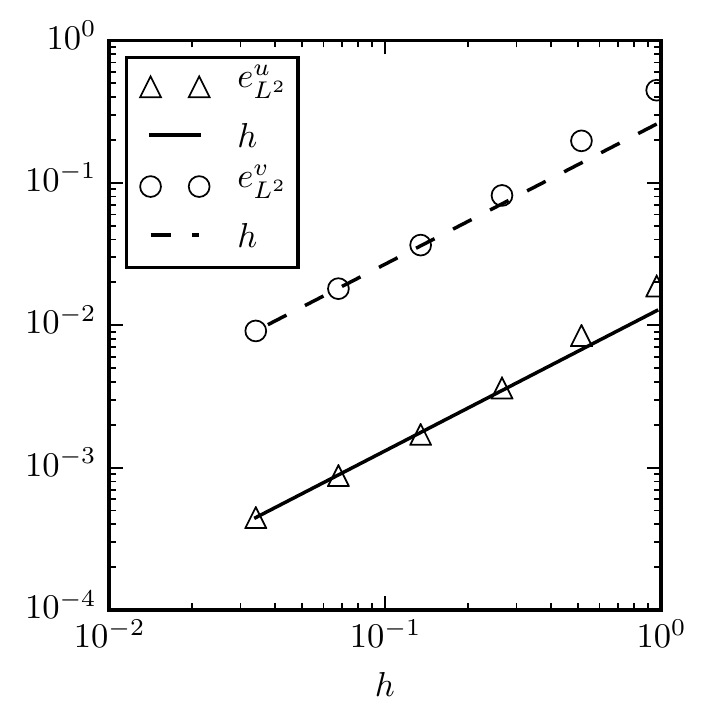}
  \includegraphics{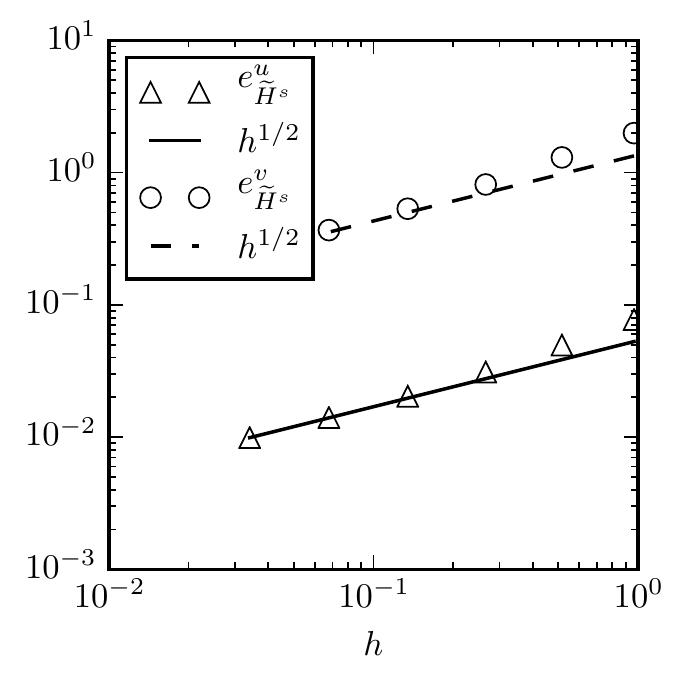}
  \caption{
    Error in \(L^{2}\)-norm (\emph{top}) and \(\widetilde{H}^{s}\left(\Omega\right)\)-norm (\emph{bottom}) in the Brusselator model.
    Optimal orders of convergence are achieved. (Compare \Cref{thm:L2conv,thm:Hsconv}.)
  }
  \label{fig:errorBrusselator}
\end{figure}

Having verified the accuracy of the method, we turn to the solution of the system \cref{eq:10,eq:11} augmented with exterior Neumann conditions as described in \Cref{sec:weak-formulation}.
Golovin, Matkowsky and Volpert \cite{GolovinMatkowskyEtAl2008_TuringPatternFormationBrusselator} observed that for \(\eta=0.2\), \(B=1.22\) and \(Q=0.1\), a single localised perturbation would first form a ring and then break up into spots.
The radius of the ring and the number of resulting spots increases as the fractional orders are decreased.
In \Cref{fig:brusselator}, simulation results for \(\alpha=\beta=0.625\) and \(\alpha=\beta=0.75\) are shown.
We observe that in both cases, an initially circular perturbation develops into a ring.
Lower diffusion coefficients do lead to a larger ring, which breaks up later and into more spots.
In the last row, we can see that the resulting spots start to replicate and spread out over the whole domain.
\begin{figure}
  \centering
  % created by hand using paraview and gimp
  \includegraphics[width=0.4\linewidth]{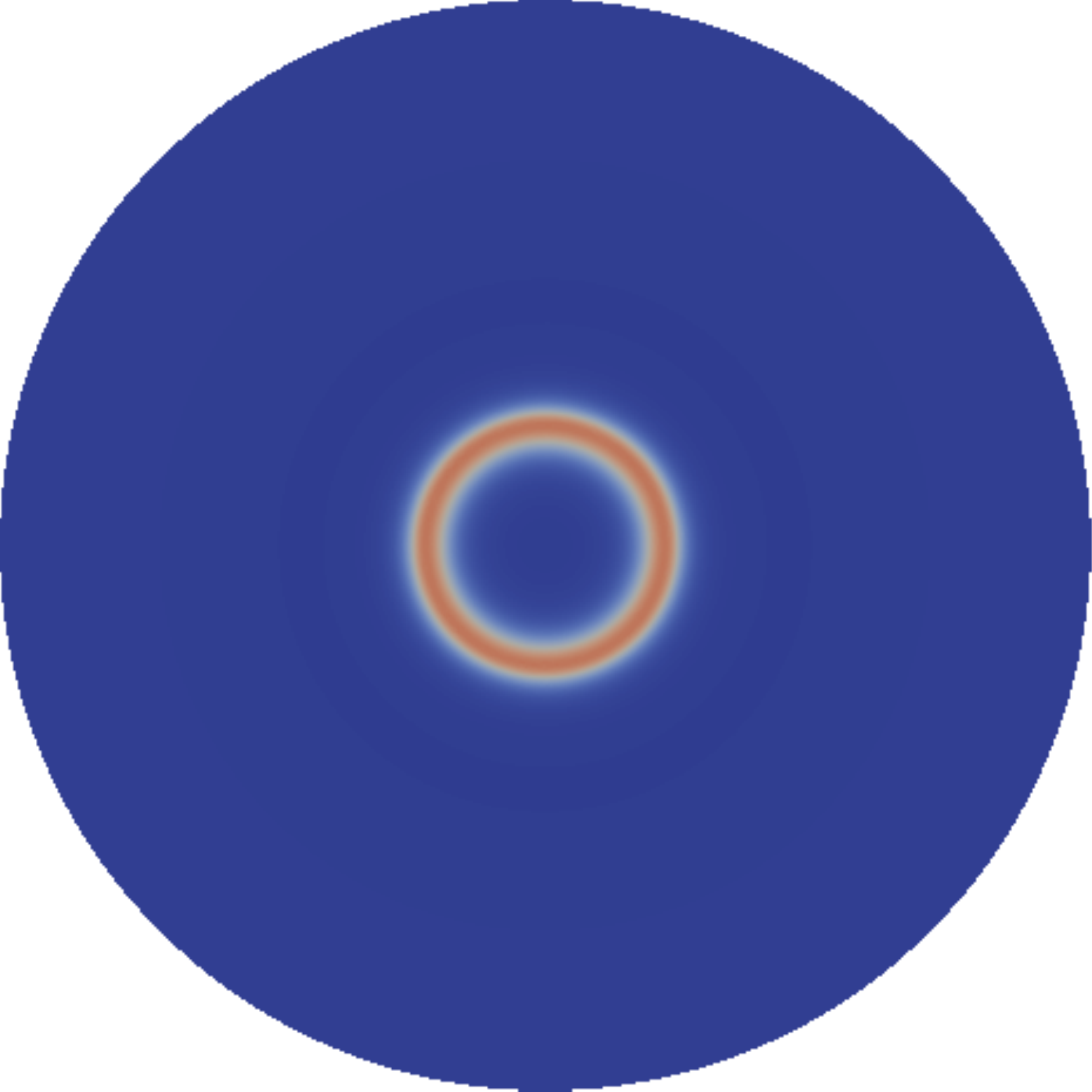}\hfil\includegraphics[width=0.4\linewidth]{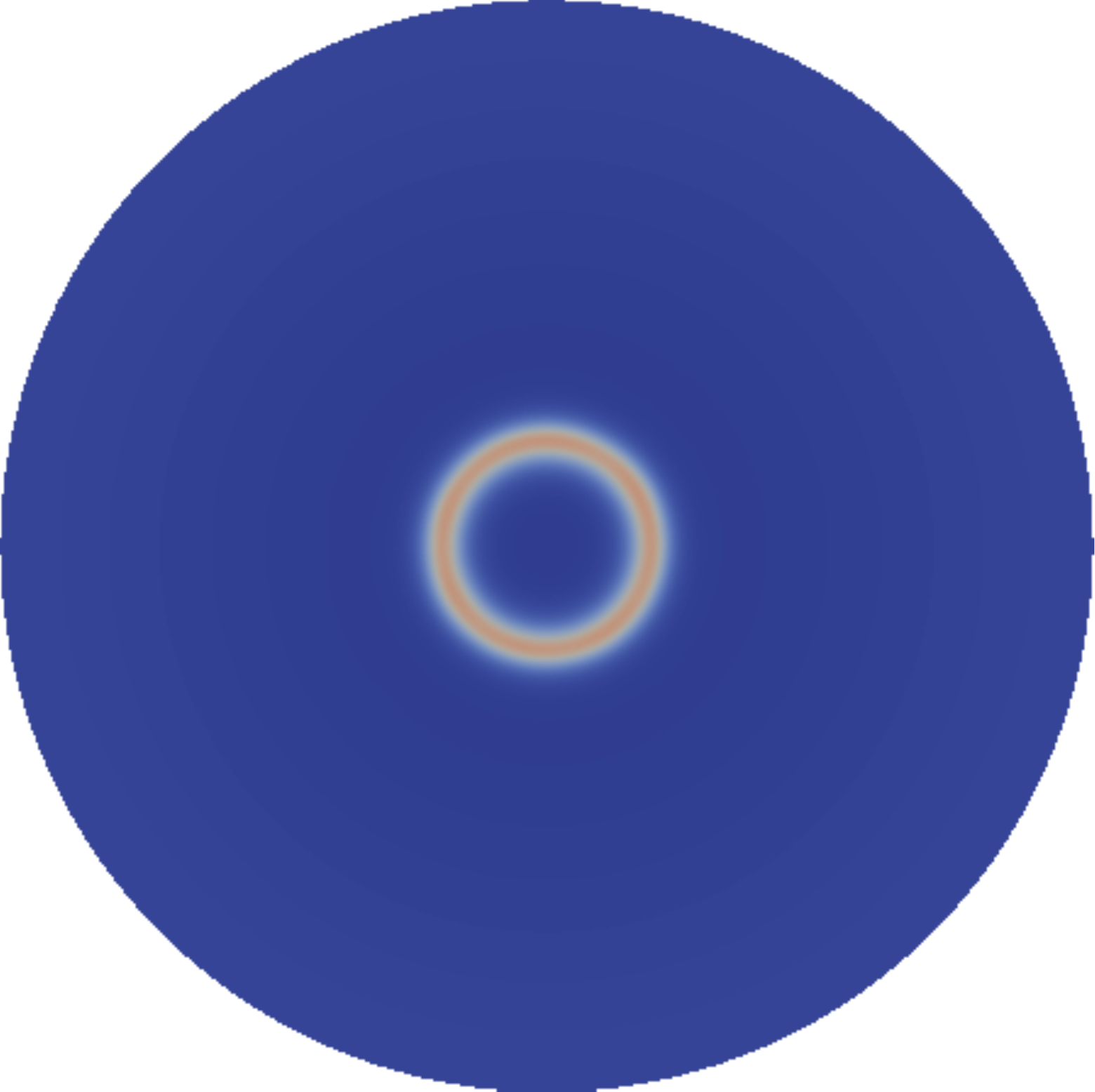}
  \includegraphics[width=0.4\linewidth]{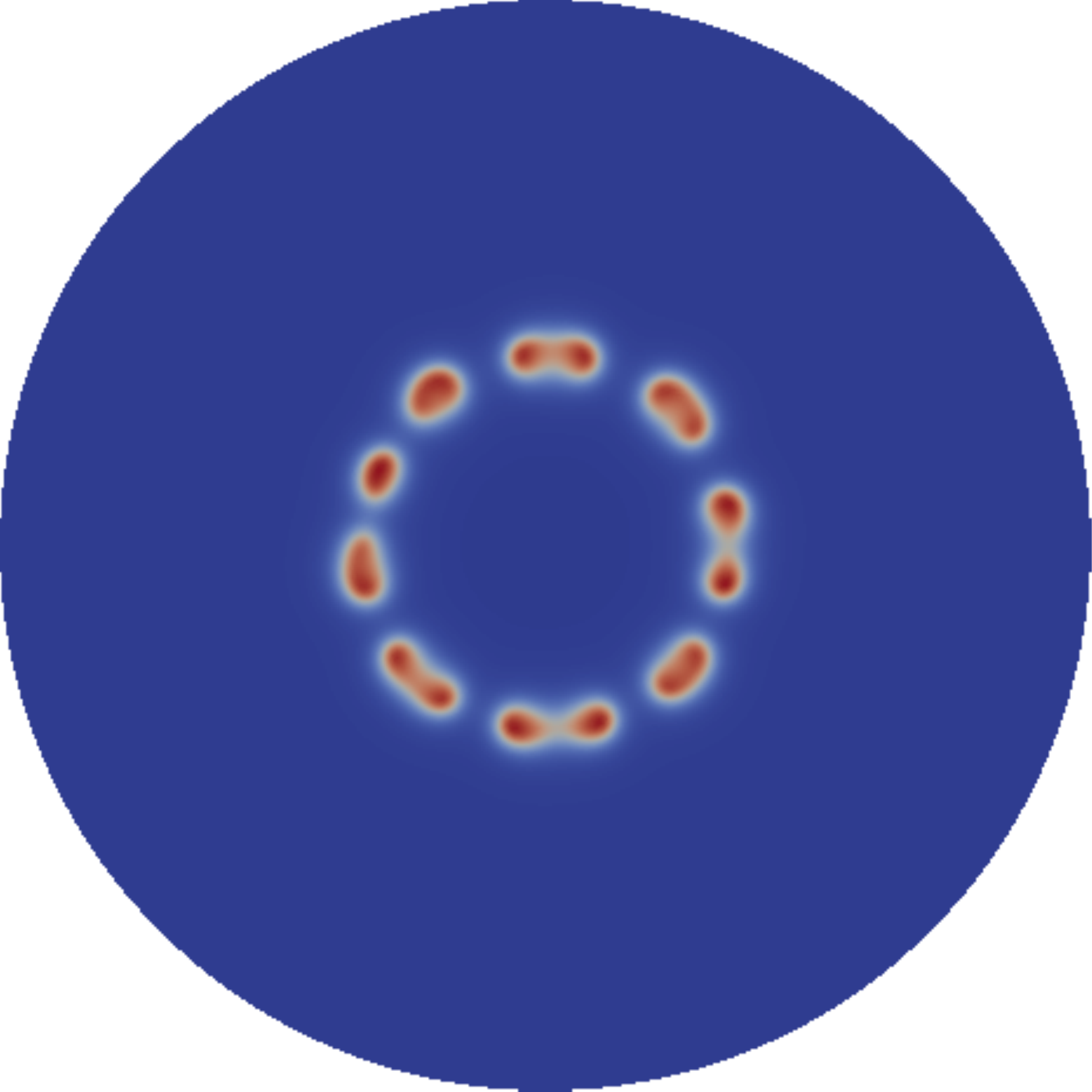}\hfil\includegraphics[width=0.4\linewidth]{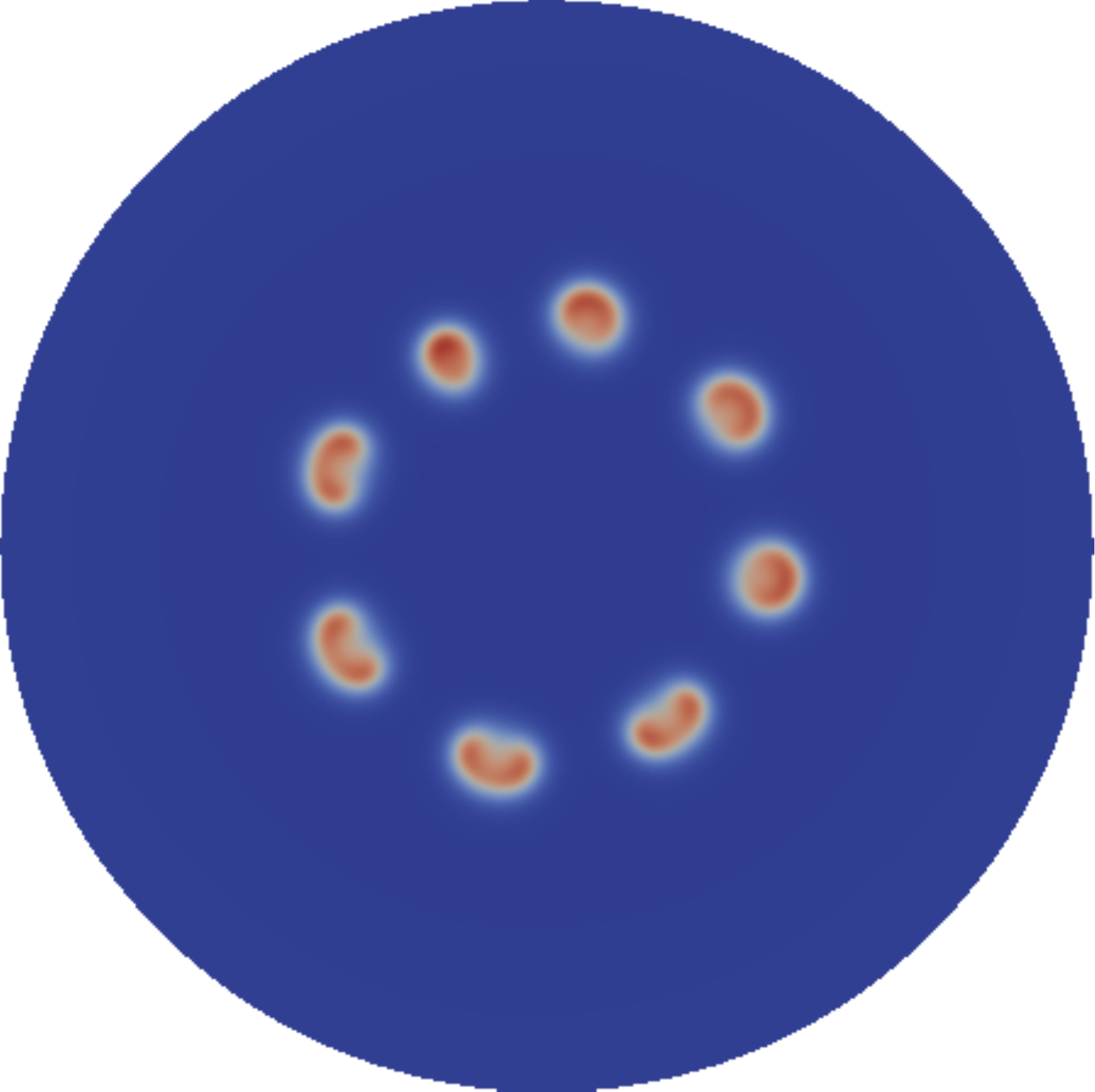}
  \includegraphics[width=0.4\linewidth]{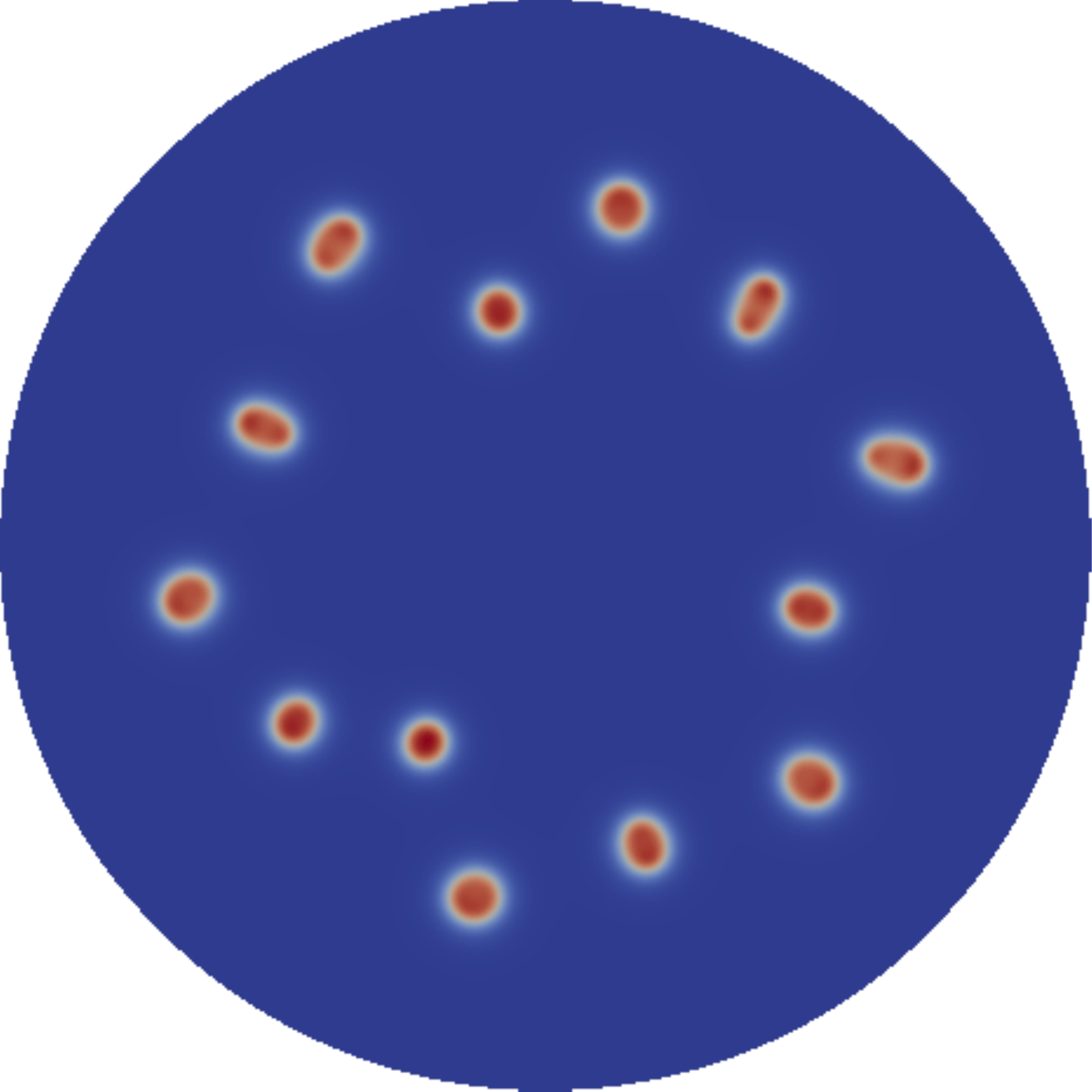}\hfil\includegraphics[width=0.4\linewidth]{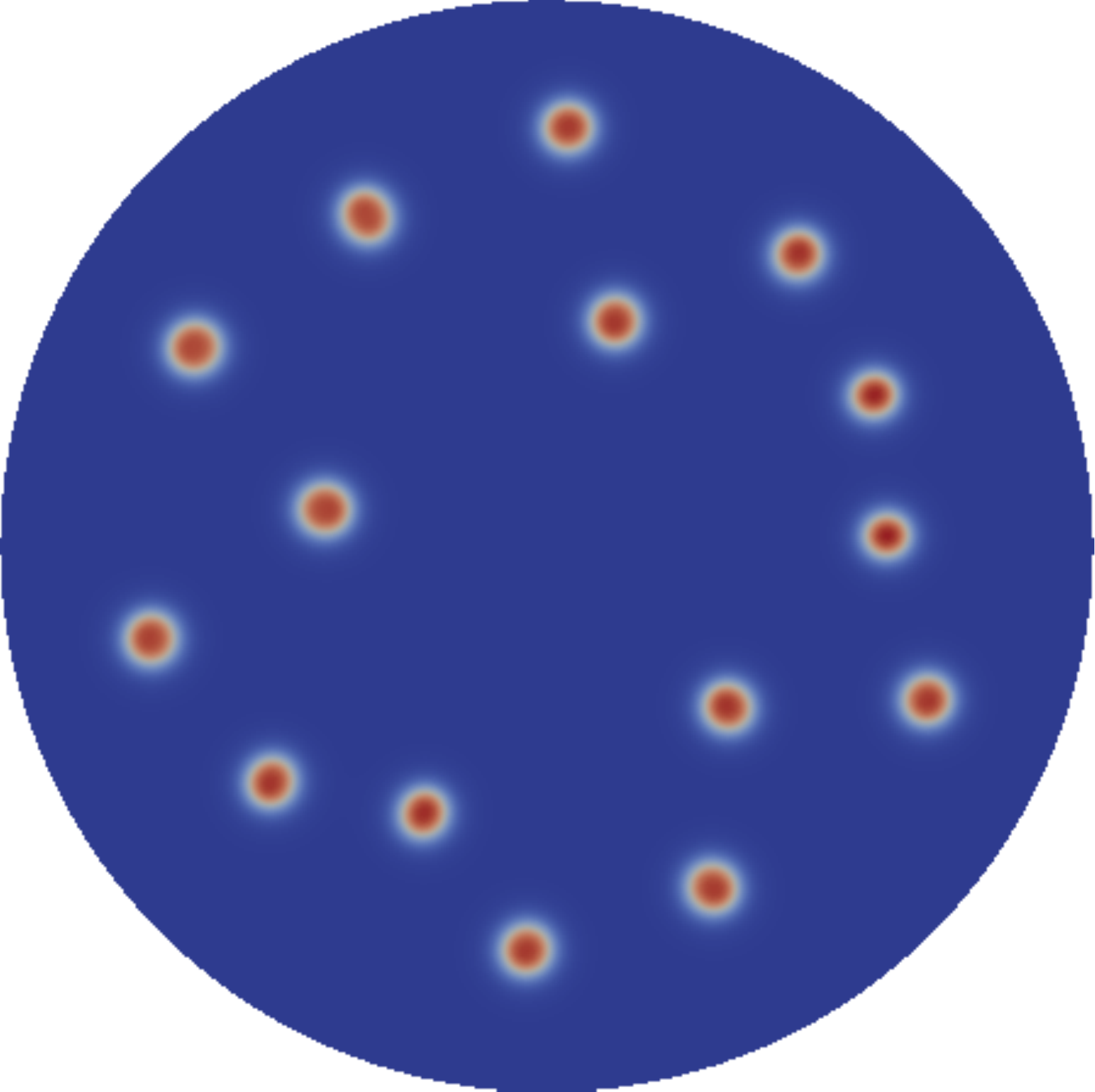}
  \caption{
    Localised spot solutions of the Brusselator system with \(\alpha=\beta=0.625\) (\emph{left}) and \(\alpha=\beta=0.75\) (\emph{right}).
    \(u\) is shown in both cases, and time progresses from top to bottom.
    The initial perturbation was identical in both cases.
    The initial perturbation in the centre of the domain forms a ring, whose radius is bigger if the fractional orders of diffusion \(\alpha\), \(\beta\) are smaller.
    The ring breaks up into several spots, which start to replicate and spread out over the whole domain.
    \(n\approx 50,000\) unknowns were used in the finite element approximation.
  }
  \label{fig:brusselator}
\end{figure}

Another choice of parameters leads to stripes in the solution.
For \(\alpha=\beta=0.75\), \(\eta=0.2\), \(B=6.26\) and \(Q=2.5\), and a random initial condition, stripes without directionality form in the whole domain.
% As time progresses, the whole domain is filled up with concentric spherical stripes.
This is in alignment with the theoretical considerations of Golovin, Matkowsky and Volpert \cite{GolovinMatkowskyEtAl2008_TuringPatternFormationBrusselator}.
\begin{figure}
  \centering
  % created by hand using paraview and gimp
  \includegraphics[width=0.4\linewidth]{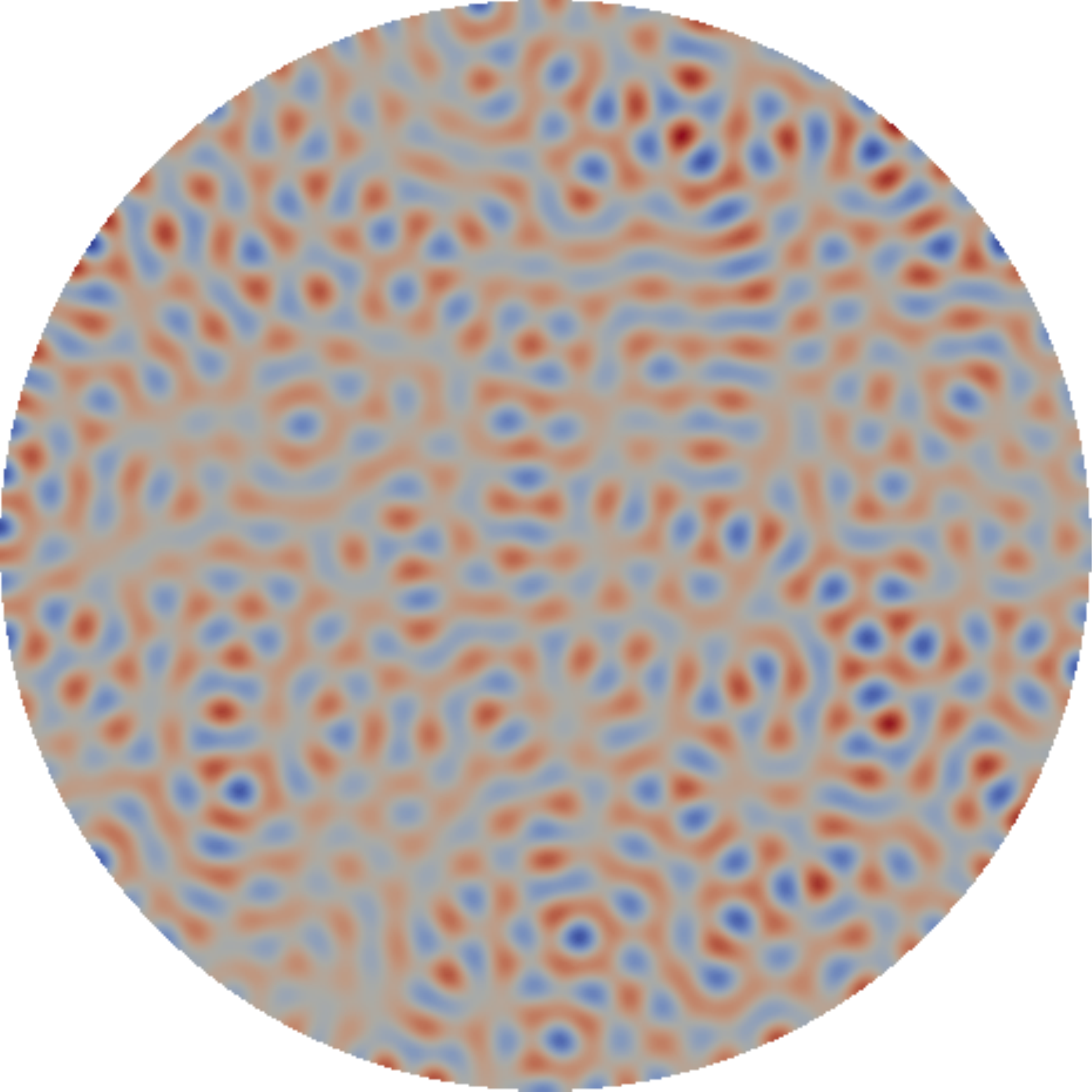}\hfil\includegraphics[width=0.4\linewidth]{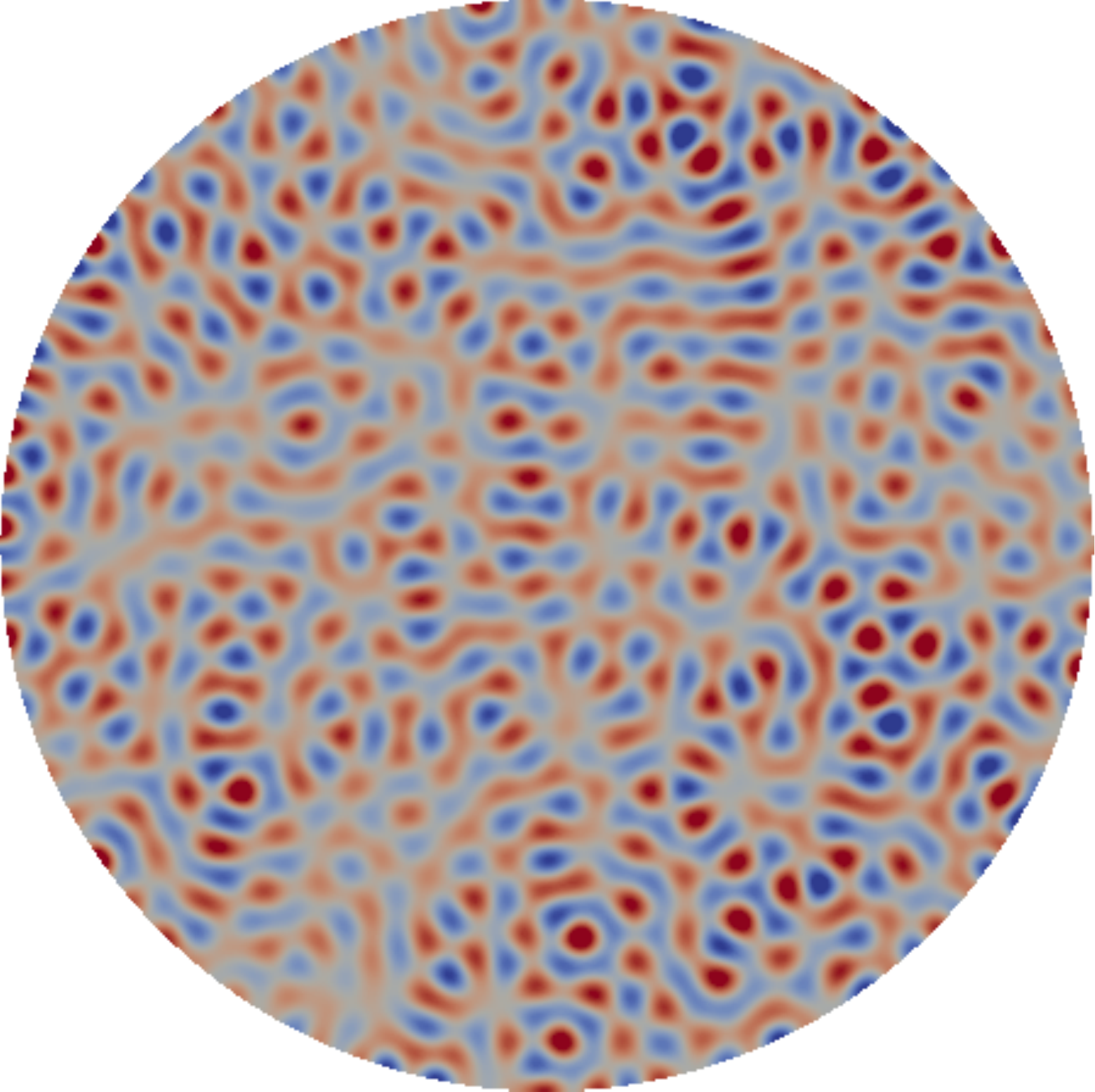}
  \caption{Stripe solutions of the Brusselator system with \(\alpha=\beta=0.75\).
    \(u\) is shown on the left, and \(v\) on the right.
    % Time progresses from top to bottom.
    The random initial condition leads to the formation of stripes throughout the domain.
    \(n\approx 50,000\) unknowns were used in the finite element approximation.
    % As time progresses, stripes align with the boundary.
    }
  \label{fig:brusselatorStripes}
\end{figure}

\section{Conclusion}
\label{sec:conclusion}

We have presented a reasonably complete and coherent approach for the efficient approximation of problems involving the fractional Laplacian, based on techniques from the boundary element literature.
In particular, we discussed the efficient assembly and solution of the associated matrix, and demonstrated the feasibility of a sparse approximation using the panel clustering method.
The potential of the approach was demonstrated in several numerical examples, and were used to reproduce some of the findings for a fractional Brusselator model.
While we focused on the case of \(d=2\) dimensions, the generalisation to higher dimensions does not pose any fundamental difficulties.
Moreover, the approach taken to obtain a sparse approximation to the dense system matrix for the fractional Laplacian does not rely strongly on the form of the interaction kernel \(k\left(\vec{x},\vec{y}\right)=\absAinsworthGlusa{\vec{x}-\vec{y}}^{-(d+2s)}\), and generalisations to different kernels such as the one used in peridynamics \cite{Silling2000_ReformulationElasticityTheoryDiscontinuities} are therefore possible.
In the present work we have confined ourselves to the discussion of quasi-uniform meshes.
However, solutions of problems involving the fractional Laplacian exhibit line singularities in the neighbourhood of the boundary.
The efficient resolution of such problems would require locally refined meshes which form the topic of forthcoming work \cite{AinsworthGlusa_AspectsAdaptiveFiniteElement}.

\appendix

\section*{Appendix A - Derivation of expressions for singular contributions}
\label{sec:deriv-expr-sing}

The contributions \(a^{K\times\tilde{K}}\) and  \(a^{K\times e}\) as given in \cref{eq:6,eq:7} for touching elements \(K\) and \(\tilde{K}\) contain removable singularities.
In order to make these contributions amenable to numerical quadrature, the singularities need to be lifted.
We outline the derivation for \(d=2\) dimensions.

The expression for \(a^{K\times \tilde{K}}\) can be transformed into integrals over the reference element \(\hat{K}\):
\begin{align*}
  &a^{K\times\tilde{K}}(\phi_{i},\phi_{j}) \\
  =& \frac{C(2,s)}{2} \int_{K} \; d \vec{x} \int_{\tilde{K}} \; d \vec{y} \frac{\left(\phi_{i}(\vec{x})-\phi_{i}(\vec{y})\right)\left(\phi_{j}(\vec{x})-\phi_{j}(\vec{y})\right)}{\absAinsworthGlusa{\vec{x}-\vec{y}}^{2+2s}} \\
  =& \frac{C(2,s)}{2}\frac{\absAinsworthGlusa{K}}{\absAinsworthGlusa{\hat{K}}} \frac{\absAinsworthGlusa{\tilde{K}}}{\absAinsworthGlusa{\hat{K}}} \int_{\hat{K}} \; d \hat{\vec{x}} \int_{\hat{K}} \; d \hat{\vec{y}}
    \frac{\left(\phi_{i}(\vec{x}\left(\hat{\vec{x}}\right))-\phi_{i}(\vec{y}\left(\hat{\vec{y}}\right))\right) \left(\phi_{j}(\vec{x}\left(\hat{\vec{x}}\right))-\phi_{j}(\vec{y}\left(\hat{\vec{y}}\right))\right)}{\absAinsworthGlusa{\vec{x}\left(\hat{\vec{x}}\right)-\vec{y}\left(\hat{\vec{y}}\right)}^{2+2s}}.
\end{align*}
Similarly, by introducing the reference edge \(\hat{e}\), we obtain
\begin{align*}
  a^{K\times e}(\phi_{i},\phi_{j})
  &= \frac{C(2,s)}{2s} \int_{K} \; d \vec{x} \int_{e} \; d \vec{y} \frac{\phi_{i}\left(\vec{x}\right) \phi_{j}\left(\vec{x}\right) ~ \vec{n}_{e}\cdot\left(\vec{x}-\vec{y}\right)}{\absAinsworthGlusa{\vec{x}-\vec{y}}^{2+2s}} \\
    &= \frac{C(2,s)}{2s} \frac{\absAinsworthGlusa{K}}{\absAinsworthGlusa{\hat{K}}} \frac{\absAinsworthGlusa{e}}{\absAinsworthGlusa{\hat{e}}} \int_{\hat{K}} \; d \hat{\vec{x}} \int_{\hat{e}} \; d \hat{\vec{y}} \frac{\phi_{i}\left(\vec{x}\left(\hat{\vec{x}}\right)\right) \phi_{j}\left(\vec{x}\left(\hat{\vec{x}}\right)\right) ~ \vec{n}_{e}\cdot\left(\vec{x}\left(\hat{\vec{x}}\right)-\vec{y}\left(\hat{\vec{y}}\right)\right)}{\absAinsworthGlusa{\vec{x}\left(\hat{\vec{x}}\right)-\vec{y}\left(\hat{\vec{y}}\right)}^{2+2s}}
\end{align*}
for touching elements \(K\) and edges \(e\).
If \(K\) and \(\tilde{K}\) or \(e\) have \(c\geq 1\) common vertices, and if we designate by \(\lambda_{k}\), \(k=0,\dots,6-c\) the barycentric coordinates of \(K\cup \tilde{K}\) or \(K\cup e\) respectively (cf. \Cref{fig:localNodes}), we have
\begin{align*}
  \lambda_{k(i)}(\hat{\vec{x}}) &= \phi_{i}\left(\vec{x}\left(\hat{\vec{x}}\right)\right),
\end{align*}
where \(k(i)\) is the local index on \(K\cup \tilde{K}\) or \(K\cup e\) of the global degree of freedom \(i\).
Moreover, we have that
\begin{align*}
  \vec{x}\left(\hat{\vec{x}}\right) - \vec{y}\left(\hat{\vec{y}}\right)
  &= \sum_{k=0}^{6-c} \lambda_{k}\left(\hat{\vec{x}}\right)\vec{x}_{k}-\sum_{k=0}^{6-c} \lambda_{k}\left(\hat{\vec{y}}\right)\vec{x}_{k} \\
  &= \sum_{k=0}^{6-c} \left[\lambda_{k}\left(\hat{\vec{x}}\right) - \lambda_{k}\left(\hat{\vec{y}}\right)\right]\vec{x}_{k}.
\end{align*}
Here, \(\vec{x}_{k}\), \(k=0,\dots,6-c\) are the vertices that span \(K\cup \tilde{K}\) or \(K\cup e\) respectively.

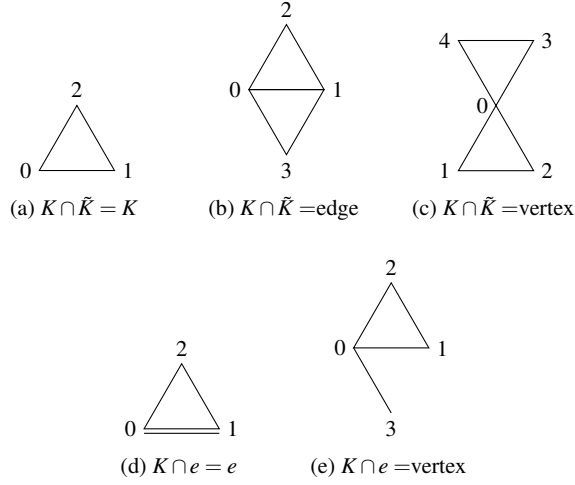
\begin{figure}
  \centering
  \subfloat[\(K\cap\tilde{K}=K\)]{
    \trimbox{-0.5cm 0cm -0.5cm 0cm}{%
      \begin{tikzpicture}
        \coordinate (node0) at (0,0);
        \coordinate (node1) at (1,0);
        \coordinate (node2) at (0.5,0.866);
        \draw (node0) -- (node1) -- (node2) -- cycle;
        \node[anchor=east] at (node0) {0};
        \node[anchor=west] at (node1) {1};
        \node[anchor=south] at (node2) {2};
      \end{tikzpicture}}}
  \subfloat[\(K\cap\tilde{K}=\)edge]{
    \trimbox{-0.5cm 0cm -0.5cm 0cm}{%
      \begin{tikzpicture}
        \coordinate (node0) at (0,0);
        \coordinate (node1) at (1,0);
        \coordinate (node2) at (0.5,0.866);
        \coordinate (node3) at (0.5,-0.866);
        \draw (node0) -- (node1) -- (node2) -- cycle;
        \draw (node0) -- (node1) -- (node3) -- cycle;
        \node[anchor=east] at (node0) {0};
        \node[anchor=west] at (node1) {1};
        \node[anchor=south] at (node2) {2};
        \node[anchor=north] at (node3) {3};
      \end{tikzpicture}}}
  \subfloat[\(K\cap\tilde{K}=\)vertex]{
    \trimbox{-0.5cm 0cm -0.5cm 0cm}{%
      \begin{tikzpicture}
        \coordinate (node1) at (0,0);
        \coordinate (node2) at (1,0);
        \coordinate (node0) at (0.5,0.866);
        \coordinate (node3) at (1,1.732);
        \coordinate (node4) at (0,1.732);
        \draw (node0) -- (node1) -- (node2) -- cycle;
        \draw (node0) -- (node3) -- (node4) -- cycle;
        \node[anchor=east] at (node0) {0};
        \node[anchor=east] at (node1) {1};
        \node[anchor=west] at (node2) {2};
        \node[anchor=west] at (node3) {3};
        \node[anchor=east] at (node4) {4};
      \end{tikzpicture}}}

  \subfloat[\(K\cap e=e\)]{
    \trimbox{-0.5cm 0cm -0.5cm 0cm}{%
      \begin{tikzpicture}
        \coordinate (node0) at (0,0);
        \coordinate (node1) at (1,0);
        \coordinate (node2) at (0.5,0.866);
        \draw (node0) -- (node1) -- (node2) -- cycle;
        \draw[transform canvas={yshift=-0.2em}] (node0) -- (node1);
        \node[anchor=east] at (node0) {0};
        \node[anchor=west] at (node1) {1};
        \node[anchor=south] at (node2) {2};
      \end{tikzpicture}}}
  \subfloat[\(K\cap e=\)vertex]{
    \trimbox{-0.5cm 0cm -0.5cm 0cm}{%
      \begin{tikzpicture}
        \coordinate (node0) at (0,0);
        \coordinate (node1) at (1,0);
        \coordinate (node2) at (0.5,0.866);
        \coordinate (node3) at (0.5,-0.866);
        \draw (node0) -- (node1) -- (node2) -- cycle;
        \draw (node0) -- (node3);
        \node[anchor=east] at (node0) {0};
        \node[anchor=west] at (node1) {1};
        \node[anchor=south] at (node2) {2};
        \node[anchor=north] at (node3) {3};
      \end{tikzpicture}}}
  \caption{Numbering of local nodes for touching triangular elements \(K\) and \(\tilde{K}\) or element \(K\) and edge \(e\).}\label{fig:localNodes}
\end{figure}

By setting
\begin{align*}
  \psi_{k}\left(\hat{\vec{x}},\hat{\vec{y}}\right) := \lambda_{k}\left(\hat{\vec{x}}\right) - \lambda_{k}\left(\hat{\vec{y}}\right),
\end{align*}
we can therefore write
\begin{align*}
  a^{K\times\tilde{K}}(\phi_{i},\phi_{j})
  &= \frac{C(2,s)}{2}\frac{\absAinsworthGlusa{K}}{\absAinsworthGlusa{\hat{K}}} \frac{\absAinsworthGlusa{\tilde{K}}}{\absAinsworthGlusa{\hat{K}}} \int_{\hat{K}} \; d \hat{\vec{x}} \int_{\hat{K}} \; d \hat{\vec{y}}  \frac{\psi_{k(i)}\left(\hat{\vec{x}},\hat{\vec{y}}\right) \psi_{k(j)}\left(\hat{\vec{x}},\hat{\vec{y}}\right)}{\absAinsworthGlusa{\sum_{k=0}^{6-c}\psi_{k}\left(\hat{\vec{x}},\hat{\vec{y}}\right) \vec{x}_{k}}^{2+2s}}.
\end{align*}
By carefully splitting the integration domain \(\hat{K}\times\hat{K}\) into \(L_{c}\) parts and applying a Duffy transformation to each part, the contributions can be rewritten into integrals over a unit hyper-cube, where the singularities are lifted.
\begin{align}
  a^{K\times\tilde{K}}(\phi_{i},\phi_{j})
  % &= \frac{C(d,s)}{2} \frac{\absAinsworthGlusa{K}}{\absAinsworthGlusa{\hat{K}}} \frac{\absAinsworthGlusa{\tilde{K}}}{\absAinsworthGlusa{\hat{K}}} \sum_{\ell=1}^{L_{c}} \int_{[0,1]^{4}} \; d \vec{\eta} ~ J^{(\ell,c)}\frac{\psi_{k(i)}^{(\ell,c)}\left(\vec{\eta}\right) \psi_{k(j)}^{(\ell,c)}\left(\vec{\eta}\right)}{\absAinsworthGlusa{\sum_{k=0}^{2d-c}\psi_{k}^{(\ell,c)}\left(\vec{\eta}\right) \vec{x}_{k}}^{d+2s}} \\
  &= \frac{C(2,s)}{2} \frac{\absAinsworthGlusa{K}}{\absAinsworthGlusa{\hat{K}}} \frac{\absAinsworthGlusa{\tilde{K}}}{\absAinsworthGlusa{\hat{K}}} \sum_{\ell=1}^{L_{c}} \int_{[0,1]^{4}} \; d \vec{\eta} ~ \Bar J^{(\ell,c)}\frac{\Bar \psi_{k(i)}^{(\ell,c)}\left(\vec{\eta}\right) \Bar \psi_{k(j)}^{(\ell,c)}\left(\vec{\eta}\right)}{\absAinsworthGlusa{\sum_{k=0}^{2d-c}\Bar\psi_{k}^{(\ell,c)}\left(\vec{\eta}\right) \vec{x}_{k}}^{2+2s}}. \label{eq:1}
\end{align}
The details of this approach can be found in Chapter 5 of \cite{SauterSchwab2010_BoundaryElementMethods} for the interactions between \(K\) and \(\tilde{K}\).
We record the obtained expressions in this case.
\begin{itemize}
\item \(K\) and \(\tilde{K}\) are identical, i.e. \(c=3\)
  \begin{align*}
     L_{3}=3,&& \Bar J^{(1,3)}=\Bar J^{(2,3)}=\Bar J^{(3,3)}= \eta_{0}^{3-2s}\eta_{1}^{2-2s}\eta_{2}^{1-2s},
  \end{align*}

  \begin{align*}
    \Bar\psi_{k}^{(1,3)}
    &= \begin{cases}
      -\eta_{3} \\
      \eta_{3}-1 \\
      1
    \end{cases}
    & \Bar\psi_{k}^{(2,3)}
    &= \begin{cases}
      -1 \\
      1-\eta_{3} \\
      \eta_{3}
    \end{cases}
    & \Bar\psi_{k}^{(3,3)}
    &= \begin{cases}
      \eta_{3} \\
      -1 \\
      1-\eta_{3}
    \end{cases}
  \end{align*}

\item \(K\) and \(\tilde{K}\) share an edge, i.e. \(c=2\)
  \begin{align*}
    L_{2}=5,
    && \Bar J^{(1,2)}=\eta_{0}^{3-2s}\eta_{1}^{2-2s}, \\
    && \Bar J^{(2,2)}=\Bar J^{(3,2)}=\Bar J^{(4,2)}=\Bar J^{(5,2)} = \eta_{0}^{3-2s}\eta_{1}^{2-2s}\eta_{2}
  \end{align*}

  \begin{align*}
    \Bar\psi_{k}^{(1,2)}
    &= \begin{cases}
      -\eta_{2} \\
      1-\eta_{3} \\
      \eta_{3} \\
      \eta_{2}-1
    \end{cases}
    & \Bar\psi_{k}^{(2,2)}
    &= \begin{cases}
      -\eta_{2}\eta_{3} \\
      \eta_{2}-1 \\
      1 \\
      \eta_{2}\eta_{3}-\eta_{2}
    \end{cases}
    & \Bar\psi_{k}^{(3,2)}
    &= \begin{cases}
      \eta_{2} \\
      \eta_{2}\eta_{3}-1 \\
      1-\eta_{2} \\
      -\eta_{2}\eta_{3}
    \end{cases} \\
    \Bar\psi_{k}^{(4,2)}
    &= \begin{cases}
      \eta_{2} \eta_{3} \\
      1-\eta_{2} \\
      \eta_{2}-\eta_{2} \eta_{3} \\
      -1
    \end{cases}
    & \Bar\psi_{k}^{(5,2)}
    &= \begin{cases}
      \eta_{2}\eta_{3} \\
      \eta_{2}-1 \\
      1-\eta_{2}\eta_{3} \\
      -\eta_{2}
    \end{cases}
  \end{align*}
\item \(K\) and \(\tilde{K}\) share a vertex, i.e. \(c=1\)
  \begin{align*}
    L_{1}=2,
    && \Bar J^{(1,1)}= \Bar J^{(2,1)}= \eta_{0}^{3-2s}\eta_{2}
  \end{align*}

  \begin{align*}
    \Bar\psi_{k}^{(1,1)}
    &= \begin{cases}
      \eta_{2}-1 \\
      1-\eta_{1} \\
      \eta_{1} \\
      \eta_{2}\eta_{3}-\eta_{2} \\
      -\eta_{2}\eta_{3}
    \end{cases}
    & \Bar\psi_{k}^{(2,1)}
    &= \begin{cases}
      1-\eta_{2} \\
      \eta_{2}-\eta_{2} \eta_{3} \\
      \eta_{2}\eta_{3} \\
      \eta_{1}-1 \\
      -\eta_{1}
    \end{cases}
  \end{align*}

\end{itemize}
We notice that the contributions for identical elements only depend on \(\eta_{3}\), so that in fact only one-dimensional integrals need to be computed.
Similarly, the cases of common edges or common vertices only require two and three dimensional integration.

In a similar fashion, the integration domain of \(a^{K\times e}\) can be split into several parts, so that the singularity can be lifted:
\begin{align*}
  a^{K\times e}(\phi_{i},\phi_{j})
  &= \frac{C(2,s)}{2s} \frac{\absAinsworthGlusa{K}}{\absAinsworthGlusa{\hat{K}}} \frac{\absAinsworthGlusa{e}}{\absAinsworthGlusa{\hat{e}}} \int_{[0,1]^{3}} \; d \vec{\eta} \Bar J^{(\ell,c)}
    \frac{\phi_{k(i)}^{(\ell,c)}\left(\vec{\eta}\right) \phi_{k(j)}^{(\ell,c)}\left(\vec{\eta}\right) ~ \sum_{k=0}^{5-c}\Bar\psi_{k}^{(\ell,c)}\left(\vec{\eta}\right) \vec{n}_{e}\cdot\vec{x}_{k}}{\absAinsworthGlusa{\sum_{k=0}^{5-c}\Bar\psi_{k}^{(\ell,c)}\left(\vec{\eta}\right) \vec{x}_{k}}^{2+2s}}.
\end{align*}
Here, \(\phi_{k}^{(\ell,c)}\) are the expressions for the local shape functions under the Duffy transformations.
The obtained expressions are
\begin{itemize}
\item \(e\) is an edge of \(K\), i.e. \(c=2\)
  \begin{align*}
    L_{2}=3,
    && \Bar J^{(1,2)}=\Bar J^{(2,2)}=\Bar J^{(3,2)}=\eta_{0}^{-2s}\left(1-\eta_{0}\right),
  \end{align*}

  \begin{align*}
    \phi_{k}^{(1,2)}
    &= \begin{cases}
      1-\eta_{0}-\eta_{2}+\eta_{0}\eta_{2} \\
      \eta_{0}+\eta_{2}-\eta_{0}\eta_{1}-\eta_{0}\eta_{1} \\
      \eta_{0}\eta_{1}
    \end{cases}
    & \phi_{k}^{(2,2)}
    &= \begin{cases}
      1-\eta_{0}-\eta_{2}+\eta_{0}\eta_{2} \\
      \eta_{2}-\eta_{0}\eta_{2} \\
      \eta_{0}
    \end{cases}\\
     \phi_{k}^{(3,2)}
    &= \begin{cases}
      1-\eta_{2}+\eta_{0}\eta_{2}-\eta_{0}\eta_{1} \\
      \eta_{2}-\eta_{0}\eta_{2} \\
      \eta_{0}\eta_{1}
    \end{cases}
  \end{align*}
  \begin{align*}
    \Bar\psi_{k}^{(1,2)}
    &= \begin{cases}
      -1 \\
      1-\eta_{1} \\
      \eta_{1}
    \end{cases}
    & \Bar\psi_{k}^{(2,2)}
    &= \begin{cases}
      -\eta_{1} \\
      \eta_{1}-1 \\
      1
    \end{cases}
    & \Bar\psi_{k}^{(3,2)}
    &= \begin{cases}
      1-\eta_{1} \\
      -1 \\
      \eta_{1}
    \end{cases}
  \end{align*}
  We notice that for \(s\geq1/2\), the integrand still contains a singularity.
  In this case, the finite element space \(V_{h}\) does not include the degrees of freedom on the boundary.
  For the interaction of the single degree of freedom that is not on the boundary (\(k=2\)), we obtain
  \begin{align*}
    \Bar J^{(1,2)}=\Bar J^{(2,2)}=\Bar J^{(3,2)}=\eta_{0}^{2-2s}\left(1-\eta_{0}\right),
  \end{align*}
  \begin{align*}
    \phi_{2}^{(1,2)}
    &= \eta_{1}
    & \phi_{2}^{(2,2)}
    &= 1
    & \phi_{2}^{(3,2)}
    &= \eta_{1}
  \end{align*}
  and \(\Bar\psi_{2}^{\ell,c}\) as above.
\item \(K\) and \(e\) share a vertex, i.e. \(c=1\)
  \begin{align*}
    L_{1}=2,
    && \Bar J^{(1,1)} = \eta_{0}^{1-2s}, \Bar J^{(2,1)}= \eta_{0}^{1-2s}\eta_{1}
  \end{align*}

  \begin{align*}
    \Bar\psi_{k}^{(1,1)}
    &= \begin{cases}
      \eta_{2}-1 \\
      1-\eta_{1} \\
      \eta_{1} \\
      -\eta_{2}
    \end{cases}
    & \Bar\psi_{k}^{(2,1)}
    &= \begin{cases}
      1-\eta_{1} \\
      \eta_{1}-\eta_{1} \eta_{2} \\
      \eta_{1}\eta_{2} \\
      -1
    \end{cases}
  \end{align*}
\end{itemize}

\section*{Appendix B - Proof of Consistency Error due to Quadrature}
\label{sec:proof-of-theorem}

Next, we give the proof for the consistency error of the quadrature approximation first stated in \Cref{sec:quadrature-errors}.
\consistencyErrorQuadrature*
\begin{proof}~
  Let the quadrature rules for the pairs \(K\times\tilde{K}\) and \(K\times e\) be denoted by \(a^{K\times\tilde{K}}_{Q}\left(\cdot,\cdot\right)\) and \(a^{K\times e}_{Q}\left(\cdot,\cdot\right)\).
  Set
  \begin{align*}
    E_{K\times\tilde{K}}^{i,j} &= a^{K\times\tilde{K}}\left(\phi_{i},\phi_{j}\right)-a_{Q}^{K\times\tilde{K}}\left(\phi_{i},\phi_{j}\right),\\
    E_{K\times e}^{i,j} &= a^{K\times e}\left(\phi_{i},\phi_{j}\right)-a_{Q}^{K\times e}\left(\phi_{i},\phi_{j}\right).
  \end{align*}
  For \(u,v\in V_{h}\), we set
  \begin{align*}
    E_{K\times\tilde{K}}(u,v)&= \sum_{i\in\mathcal{I}_{K\times\tilde{K}}} \sum_{j\in\mathcal{I}_{K\times\tilde{K}}} u_{i}v_{j}E_{K\times\tilde{K}}^{i,j}, \\
    E_{K\times e}(u,v)&= \sum_{i\in\mathcal{I}_{K}} \sum_{j\in\mathcal{I}_{K}} u_{i}v_{j}E_{K\times e}^{i,j}
  \end{align*}
  so that
  \begin{align*}
    \absAinsworthGlusa{E_{K\times\tilde{K}}(u,v)}
    &\leq \left(\max_{i,j}\absAinsworthGlusa{E_{K\times\tilde{K}}^{i,j}}\right) \sum_{i\in\mathcal{I}_{K\times\tilde{K}}} \absAinsworthGlusa{u_{i}} \sum_{j\in\mathcal{I}_{K\times\tilde{K}}} \absAinsworthGlusa{v_{j}} \\
    &\leq \left(\max_{i,j}\absAinsworthGlusa{E_{K\times\tilde{K}}^{i,j}}\right) \absAinsworthGlusa{\mathcal{I}_{K\times\tilde{K}}} \sqrt{\sum_{i\in\mathcal{I}_{K\times\tilde{K}}} \absAinsworthGlusa{u_{i}}^{2}} \sqrt{\sum_{j\in\mathcal{I}_{K\times\tilde{K}}} \absAinsworthGlusa{v_{j}}^{2}}, \\
    \absAinsworthGlusa{E_{K\times e}(u,v)}
    &\leq \left(\max_{i,j}\absAinsworthGlusa{E_{K, e}^{i,j}}\right) \sum_{i\in\mathcal{I}_{K}} \absAinsworthGlusa{u_{i}} \sum_{j\in\mathcal{I}_{K}} \absAinsworthGlusa{v_{j}} \\
    &\leq \left(\max_{i,j}\absAinsworthGlusa{E_{K, e}^{i,j}}\right) \absAinsworthGlusa{\mathcal{I}_{K}} \sqrt{\sum_{i\in\mathcal{I}_{K}} \absAinsworthGlusa{u_{i}}^{2}} \sqrt{\sum_{j\in\mathcal{I}_{K}} \absAinsworthGlusa{v_{j}}^{2}}
  \end{align*}
  Since
  \begin{align*}
    \sum_{i\in\mathcal{I}_{K\times\tilde{K}}} \absAinsworthGlusa{u_{i}}^{2}
    &\leq C\left[h_{K}^{-d}\int_{K}u^{2} + h_{\tilde{K}}^{-d}\int_{\tilde{K}}u^{2}\right], \\
    \sum_{i\in\mathcal{I}_{K}} \absAinsworthGlusa{u_{i}}^{2}
    &\leq C h_{K}^{-d}\int_{K}u^{2},
  \end{align*}
  we find
  \begin{align*}
    \absAinsworthGlusa{a(u,v)-a_{Q}(u,v)}
    &\leq \sum_{K}\sum_{\tilde{K}} \absAinsworthGlusa{E_{K\times\tilde{K}}(u,v)} + \sum_{K}\sum_{e} \absAinsworthGlusa{E_{K\times e}(u,v)} \\
    &\leq C\sum_{K}\sum_{\tilde{K}} \left(\max_{i,j}\absAinsworthGlusa{E_{K\times\tilde{K}}^{i,j}}\right) h^{-d} \left[\normAinsworthGlusa{u}_{L^{2}(K)}^{2} + \normAinsworthGlusa{u}_{L^{2}(\tilde{K})}^{2}\right]^{1/2} \\
    &\qquad\qquad \left[ \normAinsworthGlusa{v}_{L^{2}(K)}^{2} + \normAinsworthGlusa{v}_{L^{2}(\tilde{K})}^{2}\right]^{1/2} \\
    &\quad + C\sum_{K}\sum_{e} \left(\max_{i,j}\absAinsworthGlusa{E_{K\times e}^{i,j}}\right)  h^{-d}\normAinsworthGlusa{u}_{L^{2}(K)} \normAinsworthGlusa{v}_{L^{2}(K)} \\
    &\leq C h^{-d} \left(\max_{K,\tilde{K}}\max_{i,j}\absAinsworthGlusa{E_{K\times\tilde{K}}^{i,j}}\right) \sum_{K}\sum_{\tilde{K}} \normAinsworthGlusa{u}_{L^{2}(K\cup\tilde{K})} \normAinsworthGlusa{v}_{L^{2}(K\cup\tilde{K})} \\
    &\quad + C h^{-d} \left(\max_{K,e}\max_{i,j}\absAinsworthGlusa{E_{K\times e}^{i,j}}\right) \sum_{K}\sum_{e} \normAinsworthGlusa{u}_{L^{2}(K)} \normAinsworthGlusa{v}_{L^{2}(K)}.
  \end{align*}
  Because
  \begin{align*}
    \sum_{K}\sum_{\tilde{K}} \normAinsworthGlusa{u}_{L^{2}(K\cup\tilde{K})} \normAinsworthGlusa{v}_{L^{2}(K\cup\tilde{K})}
    &\leq \sqrt{\sum_{K}\sum_{\tilde{K}} \normAinsworthGlusa{u}_{L^{2}(K\cup\tilde{K})}^{2}} \sqrt{\sum_{K}\sum_{\tilde{K}} \normAinsworthGlusa{v}_{L^{2}(K\cup\tilde{K})}^{2}} \\
    &\leq 2\absAinsworthGlusa{\mathcal{P}_{h}} \normAinsworthGlusa{u}_{L^{2}(\Omega)} \normAinsworthGlusa{v}_{L^{2}(\Omega)}\\
    &\leq Ch^{-d} \normAinsworthGlusa{u}_{L^{2}(\Omega)} \normAinsworthGlusa{v}_{L^{2}(\Omega)}
      \intertext{and}
      \sum_{K}\sum_{e} \normAinsworthGlusa{u}_{L^{2}(K)} \normAinsworthGlusa{v}_{L^{2}(K)}
    &\leq \absAinsworthGlusa{\mathcal{P}_{h,\partial}} \normAinsworthGlusa{u}_{L^{2}(\Omega)} \normAinsworthGlusa{v}_{L^{2}(\Omega)}\\
    &\leq Ch^{1-d} \normAinsworthGlusa{u}_{L^{2}(\Omega)} \normAinsworthGlusa{v}_{L^{2}(\Omega)},
  \end{align*}
  we obtain
  \begin{align*}
    \absAinsworthGlusa{a(u,v)-a_{Q}(u,v)}
    &\leq C \left[h^{-2d} \left(\max_{K,\tilde{K}}\max_{i,j}\absAinsworthGlusa{E_{K\times\tilde{K}}^{i,j}}\right) \right.\\
    &\qquad\left.+ h^{1-2d}\left(\max_{K,e}\max_{i,j}\absAinsworthGlusa{E_{K\times e}^{i,j}}\right)\right] \normAinsworthGlusa{u}_{L^{2}(\Omega)} \normAinsworthGlusa{v}_{L^{2}(\Omega)}.
  \end{align*}
  For \(d=2\), using \Cref{thm:localError} stated below permits to conclude.
\end{proof}

\begin{theorem}[\cite{SauterSchwab2010_BoundaryElementMethods}, Theorems 5.3.23 and 5.3.24]\label{thm:localError}
  If \(K\) and \(\tilde{K}\) (\(K\) and \(e\)) are touching elements, then
  \begin{align*}
    \absAinsworthGlusa{E_{K\times\tilde{K}}^{i,j}}&\leq Ch^{2-2s}\rho_{1}^{-2k_{T}},\\
    \absAinsworthGlusa{E_{K\times e}^{i,j}} &\leq Ch^{2-2s}\rho_{3}^{-2k_{T,\partial}},
  \end{align*}
  where \(\rho_{1},\rho_{3}>1\) and \(k_{T}\), \(k_{T,\partial}\) are the quadrature orders in every dimension of \cref{eq:8,eq:9}.

  If \(K\) and \(\tilde{K}\) (\(K\) and \(e\)) are not touching, then
  \begin{flalign*}
    \absAinsworthGlusa{E_{K\times\tilde{K}}^{i,j}} &\leq C h^{2} d_{K,\tilde{K}}^{-2s}\tilde{\rho}_{2}\left(K,\tilde{K}\right)^{-2k_{NT}},\\
    \absAinsworthGlusa{E_{K\times e}^{i,j}} &\leq C h^{2} d_{K,e}^{-2s}\tilde{\rho}_{4}\left(K,e\right)^{-2k_{NT.\partial}},
  \end{flalign*}
  where \(d_{K,\tilde{K}}:=dist(K,\tilde{K})\), \(d_{K,e}:=dist(K,e)\), \(\tilde{\rho}_{2}(K,\tilde{K}):=\rho_{2}\max\left\{\frac{d_{K,\tilde{K}}}{h},1\right\}\), \(\tilde{\rho}_{4}(K,\tilde{K}):=\rho_{4}\max\left\{\frac{d_{K,e}}{h},1\right\}\) and \(\rho_{2},\rho_{4}>1\), and \(k_{NT}\), \(k_{NT,\partial}\) are the quadrature order in every dimension of \cref{eq:6,eq:7}.
\end{theorem}

% \bibliographystyle{spmpsci}
% \bibliography{bibtex/papers.bib}

\end{document}